\documentclass[10pt]{article}
\usepackage{amsmath,amssymb,amsthm,amscd}

\allowdisplaybreaks[1]

\usepackage{graphicx}

\newcommand{\N}{\mathbb{N}}
\newcommand{\Z}{\mathbb{Z}}

\newcommand{\C}{\mathbb{C}}

\newtheorem{defn}{Definition}[section]
\newtheorem{thm}[defn]{Theorem}

\newtheorem{prop}[defn]{Proposition}
\newtheorem{cor}[defn]{Corollary}
\newtheorem{lemma}[defn]{Lemma}

\newcommand{\nr}[1]{\left\Vert #1\right\Vert}
\newcommand{\abs}[1]{\left\vert #1\right\vert}
\newcommand{\restr}[2]{\left.#1\right|_{#2}}

\newcommand{\eps}{\varepsilon}
\renewcommand{\phi}{\varphi}

\DeclareMathOperator{\Span}{span}
\DeclareMathOperator{\Tr}{Tr}

\title{Spectral triples for subshifts}
\author{Antoine Julien\\
Norwegian University of Science and Technology, \\ Trondheim, Norway\\ antoine.julien@math.ntnu.no, 
\and
Ian F. Putnam,\thanks{Supported in part by a grant from NSERC, Canada}\\
Department of Mathematics and Statistics,\\
University of Victoria,\\
Victoria, B.C., Canada V8W 3R4,\\
ifputnam@uvic.ca}

\begin{document}
\maketitle

\begin{abstract}
We propose a construction for spectral triple on algebras associated with subshifts.
One-dimensional subshifts provide concrete examples of $\Z$-actions on Cantor sets.
The $C^*$-algebra of this dynamical system is generated by functions in $C(X)$ and a
unitary element $u$ implementing the action.
Building on ideas of Christensen and Ivan, we give a construction of a family of spectral
triples on the commutative algebra $C(X)$.
There is a canonical choice of eigenvalues for the Dirac operator $D$ which ensures that $[D,u]$
is bounded, so that it extends to a spectral triple on the crossed product.

We study the summability of this spectral triple, and provide examples for which the Connes
distance associated with it on the commutative algebra is unbounded, and some
for which it is bounded.
We conjecture that our results on the Connes distance extend to the spectral triple defined
on the noncommutative algebra.
\end{abstract}

\newpage

\section{Introduction}

Alain Connes introduced the notion of a spectral triple
which consists of a separable, infinite dimensional
complex Hilbert space $\mathcal{H}$, a $*$-algebra 
of bounded linear operators on $\mathcal{H}$, $A$,
and a self-adjoint unbounded operator $D$ such that
$( 1 + D^{2})^{-1}$ is a compact operator.
The key relation is that for all 
$a$ in $A$, $[D, a] = D a - a D$ is densely defined
and extends to a bounded operator~\cite{Co89}.
The motivating example is where $A$ consists
of the smooth functions on a compact 
manifold and $D$ is some type of elliptic 
differential operator.

In spite of its geometric origins, there has been 
considerable interest in finding examples where
the algebra $A$ is the continuous functions 
on a compact, totally disconnected metric space with 
no isolated points. We refer to such a space as a Cantor set.
The first example was given by Connes~\cite{Co85, Co89, Con94}
but many other authors have also contributed:
Sharp~\cite{Sha12}, Pearson--Bellissard~\cite{PB09},
Christensen--Ivan~\cite{CI06}, etc.
Many of these results concern not just a Cantor set, 
but also some type of dynamical system on a Cantor set.
Many such systems are closely related to 
aperiodic tilings and are important as mathematical models
for quasicrystals.

We continue these investigations here by studying
subshifts.  Let $\mathcal{A}$ be a finite set 
(called the alphabet). Consider $\mathcal{A}^{\mathbb{Z}}$
with the product topology and the homeomorphism $\sigma$
which is the left shift: $\sigma(x)_{n} = x_{n+1}, 
n \in \mathbb{Z}$, for
all $x$ in  $\mathcal{A}^{\mathbb{Z}}$.
A subshift $X$ is any non-empty, closed and shift-invariant subset
 of  $\mathcal{A}^{\mathbb{Z}}$.
Is is regarded as a dynamical system for the map $\sigma|_{X}$.

Our aim here is to construct and study examples
of spectral triples for the algebra of locally 
constant functions on $X$, which we denote 
$C_{\infty}(X)$. Besides, we let $C(X)$ denote the $C^{*}$-algebra 
of continuous functions on $X$ and $C(X) \times \mathbb{Z}$ be the 
crossed product $C^*$-algebra generated by $C(X)$ and a canonical unitary $u$.
We will also consider the $*$-subalgebra of the crossed product 
generated by $C_{\infty}(X)$ and $u$ and study how our spectral triples extend to this algebra.

Our construction is just a special case of 
that given by Christensen and Ivan for 
AF-algebras \cite{CI06}. Given an increasing 
sequence of finite-dimensional 
$C^{*}$-algebras
\[
A_{1} \subseteq A_{2} \subseteq A_{3} \subseteq \cdots
\]
one considers their union and a representation
constructed by the GNS method. The same sequence
of sets provides an increasing sequence of finite-dimensional
subspaces for the associated Hilbert spaces which is used
the construct the eigenspaces for the operator $D$.
The one other piece of data needed for our construction
is a $\sigma$-invariant measure with support $X$, 
which we denote $\mu$.

Our case of the subshift has some special features. The 
first is that the sequence of finite-dimensional 
subalgebras we construct is canonical, arising
from the structure as a subshift. Secondly, the 
choice of eigenvalues for the operator $D$ is 
constrained by the dynamics (see Theorem 3.2)
in such a way that we find a canonical choice for $D$,
which we denote by $D_{X}$.

One similarity with the situation of Christensen and Ivan 
is that our operator $D_{X}$ is positive, meaning that,
at least to this point, it has no interesting index data,
but is rather regarded as a noncommutative metric structure
as in Rieffel's work~\cite{Rie99, Rie04}.

We mention one other result by Bellissard, Marcolli and 
Reihani~\cite{BMR10} which is relevant and 
motivated our work. They begin with a spectral triple
$( \mathcal{H}, A, D)$ and an action $\alpha$ 
of the integers on $A$ by automorphisms. Under the 
assumption that the action is quasi-isometric 
(that is,   there is a constant $C \geq 1$ such that
\[
C^{-1} \Vert [D, a] \Vert \leq \Vert 
[D, \alpha^{n}(a)] \Vert \leq C \Vert 
[D, a] \Vert,
\]
for all $a$ in $A$ and $n $ in $\mathbb{Z}$, they construct
a spectral triple for $A \times \mathbb{Z}$
satisfying certain conditions. Moreover, they show that 
every spectral triple on $A \times \mathbb{Z}$ satisfying 
these conditions arises this way. While this is 
a very good result, it is somewhat disappointing
since the hypothesis of quasi-isometric is very strong
and quite atypical in dynamical systems of interest.
In fact, subshifts are expansive. That is, if
$d$ is any metric on the shift space,
there is a constant $\epsilon_{0} > 0$ such that for
every $x \neq y$, there is an 
integer $n$ with $d(\sigma^{n}(x), \sigma^{n}(y)) \geq \epsilon_{0}$. In some sense, this is the opposite 
of quasi-isometric. Of course,  our spectral triples do not satisfy 
the conditions given in~\cite{BMR10}. The key point
is that the representation of $A \times \mathbb{Z}$ used
in~\cite{BMR10} is the left regular one on
$\mathcal{H} \otimes \ell^{2}(\mathbb{Z})$. Our representations 
are on the  Hilbert space $L^{2}(X, \mu)$, which 
may be regarded as more natural from a dynamical perspective.

We  also mention previous constructions of spectral triples for
subshifts or tiling spaces (see also the review article~\cite{JKS15}).
Kellendonk, Lenz and Savinien~\cite{KSav12, KLS13} built spectral
triples for the commutative algebra of one-dimensional subshifts.
Their construction is similar to ours in the sense that it is based
on the ``tree of words'' (which is related to our increasing family
of subalgebras).
In their papers, however, the algebra is represented on a Hilbert space
of the form $\ell^2(E)$, where $E$ is a countable set built from a discrete
approximation of the subshift. It is
unclear how the shift action reflects on this set, and there doesn't seem
to be an obvious extension to the noncommutative algebra.
They can still relate combinatorial properties of the subshift (namely
the property of being repulsive) to properties of the spectral triples.
Their construction also seems to be topologically quite rich, as the $K$-homology
class of their triple is in general not trivial.
In a recent paper, Savinien extended some results to higher-dimensional tilings~\cite{Sav15}.
Previously, Whittaker~\cite{Whi13} had proposed a spectral triple construction
for some hyperbolic dynamical systems (Smale spaces), a family which contains
self-similar tiling spaces.
In Whittaker's construction, the triple is defined on a noncommutative algebra
which corresponds to $C(X) \times \Z$ in our case. His representation
is however quite specific to the Smale space structure, and it is unclear
how our construction compares to his.

This paper is organized as follows. In the next section,
we discuss the basics of subshifts. This includes an 
introduction to three important classes which 
will be our main examples. The third section gives the construction of the spectral triples and 
some of their basic properties. In the fourth section,
we establish results on the summability of the 
spectral triples. Leaving the precise statements 
until later, various aspects of summability are closely
linked with the entropy of the subshift and also
its complexity. In the fifth section, we discuss 
the Connes metric. At this point all of our results
here deal with the algebra $C_{\infty}(X)$. We hope
to extend these to $C_{\infty}(X) \times \mathbb{Z}$
in future work. Again leaving the precise statements
until later, it turns out that the behaviour
of the Connes metric depends on very subtle
properties of the dynamics and varies quite wildly between
our three classes of examples.

\subsection*{Summary of the main results}

Given a subshift $X$ with an invariant measure $\mu$, we represent $C(X) \times \Z$ on $L^2(X,\mu)$. Let $C_\infty(X)$ be the algebra of locally constant functions on $X$.
For any increasing sequence going to infinity $(\alpha_n)_{n \geq 0}$, we define a Dirac operator with these eigenvalues, so that $(C_\infty(X),L^2(X,\mu),D)$ is a spectral triple (Theorem~\ref{thm:spectral-triple}). 
It extends to a spectral triple on $C_\infty(X) \times \Z$ if and only if $(\alpha_n - \alpha_{n-1})$ is bounded. This leads to the definition of a canonical operator $D_X$ (with $\alpha_n = n$), which is studied in the rest of the paper (Definition~\ref{def:DX}).

The summability depends on the complexity of the subshift. The complexity counts the number of factors of length $n$ appearing in the subshift. If the subshift has positive entropy $h(X)$ (which corresponds to the complexity function growing exponentially), the summability of $e^{sD_X}$ depends on whether $s < h(X)$ or $s > h(X)$ (Theorem~\ref{thm:theta-summable}). In the case of zero entropy, the complexity may grow asymptotically like $n^d$ for some number $d$. Whether $D_X$ is $s$-summable then depends whether $s < d$ or $s > d$ (Theorem~\ref{thm:finite-summable}).

In the last section, we investigate the Connes metric for various subshifts. In this section, we focus on the spectral triple restricted to the commutative algebra $C_\infty (X)$.
For (aperiodic, irreducible) shifts of finite type, the Connes metric is infinite (Theorem~\ref{thm:connes-SFT}).
For linearly recurrent subshifts, the Connes metric is finite and induces the weak-$*$ topology (Theorem~\ref{thm:LR}). This result is 
positive in the sense that it includes many examples 
of interest related to one-dimensional quasicrystals, including 
primitive substitution subshifts.
For Sturmian subshifts, 
the results depend rather subtly (but probably not surprisingly)
on the continued fraction
expansion of the irrational number which parameterizes the
subshift.
We show that there is a large class of
 numbers $\theta \in (0,1)$ such that the subshift with parameter $\theta$ has finite Connes metric and induces the 
 weak-$*$ topology (Theorem~\ref{thm:sturm-bounded-ae}). 
 In particular, this is the case 
 for almost all $\theta \in (0,1)$ (for the Lebesgue measure). However, we also exhibit Sturmian subshifts having
  infinite Connes metric (Theorem~\ref{thm:sturm-unbounded}).

\subsubsection*{Acknowledgments}
Substantial part of this work was done while A.~Julien was working at the University of Victoria, funded in part by the Pacific Institute for the Mathematical Sciences (PIMS). We also thank the referee for a
 thorough reading of the paper and many helpful 
 suggestions.

\section{Subshifts}
\label{sec:subshifts}

Let $\mathcal{A}$ be a finite set. We let 
$\sigma: \mathcal{A}^{\Z} \rightarrow  \mathcal{A}^{\Z} $ denote the left shift
map. That is, $\sigma(x)_{n} = x_{n+1}$, for $x$ in $\mathcal{A}^{\Z} $ and $n$ in $\Z$.
Let $X \subset \mathcal{A}^{\Z}$ be a subshift; that is, it is a closed subset with $\sigma(X) = X$.
Let $\mu$ be a $\sigma$-invariant probability measure on $X$ with support equal to $X$. The system $(X, \sigma)$ always admits 
invariant measures, but the hypothesis that it has full support
is non-trivial. It 
is satisfied in most cases of interest.

For $n \geq 1$, we let $X_{n}$ denote the words in $X$ of length $n$. That is, 
we have 
\[
 X_{n} = \{ (x_{1}, x_{2}, \ldots, x_{n}) \mid x \in X \}.
\]
We say that such finite words are \emph{admissible}.
For convenience, we let $X_{0}$ be the set consisting of the empty word, which we denote
$\varepsilon$.
For $w$  in $X_{n}$, we define the associated cylinder set
\begin{align*}
  U(w) & =  \{ x \in X \mid (x_{-m}, \ldots, x_{m}) = w \}, n = 2m+1, \\
  U(w) & =  \{ x \in X \mid (x_{1-m}, \ldots, x_{m}) = w \}, n = 2m.
  \end{align*}

We let $\xi(w)$  be the characteristic function  of $U(w)$. The set $U(w)$ is clopen and 
$\xi(w)$ lies in both $C(X)$ and $L^{2}(X, \mu)$. We let $\mu(w) = \mu(U(w))$, for 
convenience. The fact that we assume the support of $\mu$ is $X$ means that $\mu(w) \neq 0$, for 
any word $w$. 

We define $\pi : X_{n} \rightarrow X_{n-1}$, for $n > 1$, by 
\begin{align*}
  \pi(w_{-m}, \ldots, w_{m})   & =  (w_{1-m}, \ldots, w_{m})   & & \text{if } n = 2m+1, \\
  \pi(w_{1-m}, \ldots, w_{m})  & =  (w_{1-m}, \ldots, w_{m-1}) & & \text{if } n = 2m.
\end{align*}
We extend this to $n=1$ by defining $\pi(a)$ to be the empty word, for all $a$ in $X_{1}$.
We also define $U(\varepsilon)$ to be $X$.
Notice with this indexing, we have $U(\pi(w)) \supset U(w)$, for all $w$ in $X_{n}$, $n > 1$.

Of course, the function $\pi$ may be iterated and for all positive integers $k$, 
$\pi^{k} : X_{n} \rightarrow X_{n-k}$. This leads to the unfortunate notation
$\pi^{n-m} : X_{n} \rightarrow X_{m}$, when $n > m$. For convenience,
we let $\pi_{m}$ be the map defined on the union of all $X_{n}, n > m$, which
is $\pi^{n-m}$ on $X_{n}$. Thus $\pi_{m}: \cup_{n > m} X_{n} \rightarrow X_{m}$.
We also extend this function to be defined on $X$ by 
\begin{align*}
  \pi_{2m+1}(x) & = (x_{-m}, \ldots, x_{m}), &
  \pi_{2m}(x)   & = (x_{1-m}, \ldots, x_{m}).
\end{align*}

Observe that if $w$ is in $X_{n}$ then $\# \pi^{-1}\{ w \} > 1$ if and only if, for any
$w'$ with $\pi(w') = w$, we have $U(w) \supsetneq U(w')$. For 
 $n \geq 1$, we denote by $X^{+}_{n}$ the $w$ in $X_{n}$ satisfying either condition.
Such words are called (left-\ or right-) \emph{special words}.

Of basic importance in the dynamics of subshifts is the \emph{complexity} of the shift
$X$, which is simply $\# X_{n}, n \geq 1$, regarded as a function on the natural numbers.
Notice first that if $X = \mathcal{A}^{\Z}$ is the full shift, then the number of words
of length $n$ is simply $(\# \mathcal{A})^{n}$, for $n \geq 1$. More generally, if $X$ is a somewhat large subset
of $ \mathcal{A}^{\Z}$, then we reasonably expect the growth of the complexity to 
be exponential. To capture the base of the exponent, the \emph{entropy}  of 
$X$, denoted $h(X)$, is defined to be
\begin{equation}\label{eq:entropy}
 h(X) = \lim_{n \rightarrow \infty} \frac{\log ( \# X_{n})}{n}
     = \inf \left\{  \frac{\log ( \# X_{n})}{n} \ ; \ n \geq 1 \right\}.
\end{equation}
Of course, it is a non-trivial matter to see the limit exists, and equals the infimum.
This ``configuration entropy'' or ``patch-counting entropy'' agrees with the topological entropy for the dynamical system. See~\cite{LP03,BLR07}.
Secondly, there is some variation in the literature on dynamical systems as to the base with which the logarithm should be taken.
We will use $\log$ to mean natural logarithm throughout.

We also define $R(w)$ for any $w$ in $\cup_{n \geq 1} X_{n}$
\[
 R(w) = \sup \Bigl\{ \frac{\mu(w)}{\mu(w')} \mid \pi(w') = w \Bigr\}.
\]
Notice that $R(w) \geq 1$, with strict inequality if and only 
if $w$ is in $\cup_{n} X^{+}_{n}$.

\begin{lemma}
\label{lemma:Rw}
 For any $w$ in $X^{+}_{n}, n \geq 1$ and $w'$ with $\pi(w') = w$, we have  
 \[
  \left( \frac{\mu(w)}{\mu(w')} - 1 \right)^{-1} \leq R(w).
 \]
\end{lemma}

\begin{proof}
Let $\pi^{-1}\{ w \} = \{ w_{1}, \ldots, w_{I} \}$, arranged so that \newline 
$ \mu(w_{1}) \leq \mu(w_{2}) \leq  \ldots \leq \mu(w_{I})$. Since $w$ is in $X^{+}_{n}$, 
we know that $I \geq 2$. Since $U(w)$ is the disjoint union of the sets $U(w_{i}), 1 \leq i \leq I$,
we have $\mu(w_{1}) + \mu(w_{2}) +  \ldots + \mu(w_{I}) = \mu(w)$.
It follows that for any $ 1 \leq i \leq I$, we have 
\[
 \left( \frac{\mu(w)}{\mu(w_{i})} - 1 \right)^{-1} = \frac{\mu(w_{i})}{\mu(w) - \mu(w_{i})} \leq 
   \frac{\mu(w)}{ \mu(w_{1})} = R(w).
\]
\end{proof}

\subsection{Shifts of finite type}
\label{ssec:sft}

Let $G = (G^{0}, G^{1}, i, t)$ be a finite directed graph. That is, $G^{0}, G^{1}$ are finite sets 
(representing the vertices and edges, respectively) and 
$i,t$ (meaning 'initial' and 'terminal') are maps from $G^{1}$ to $G^{0}$.

The set of edges in the graph is our symbols, $\mathcal{A} = G^{1}$, and $X^{G}$ consists of all 
doubly infinite paths in the graph, \emph{i.e.} 
\[
 X^{G} = \{ ( x_{n})_{n \in \Z} \mid x_{n} \in G^{1}, t(x_{n}) = i(x_{n+1}), \text{ for all } n \in \Z \}.
\]
This set $X^G$ is a subshift of the set $\bigl(G^1 \bigr)^\Z$ and is called a (nearest-neighbour)
\emph{subshift of finite type}, or simply \emph{shift of finite type}, or SFT.

Associated to the graph $G$ is an incidence matrix $A$. It is simplest to describe if we
set $G^{0} = \{ 1, 2, \ldots, I\}$, for some positive integer $I$. In tis case, $A$ is 
an $I \times I$ matrix and 
\[
 A_{i,j} = \# \{ e \in G^{1} \mid i(e)=i, t(e) = j \}, \  1 \leq i, j \leq I.
\]
For any $n \geq 1$, $1 \leq i, j \leq I$, $A^{n}_{i,j}$ equals the number of paths of length $n$
from $i $ to $j$. 

We say that $G$ is \emph{irreducible} if  for each $i,j$, there exists $n$ with $A^{n}_{i,j} > 0$.
We say that $G$ is aperiodic if $X^{G}$ is infinite. For an irreducible, aperiodic
shift of finite type, the matrix $A$ has a unique 
positive eigenvalue of maximum absolute value called the Perron eigenvalue, which we denote
by $\lambda_{G}$.
We denote $u_G$ and $v_G$ (or simply $u$ and $v$) respectively the left and right eigenvectors
associated with $\lambda_G$, normalised such that $\sum_k v_k = 1$ and $\sum_k u_k v_k = 1$. The subshift has a unique invariant measure of maximal entropy (see for example~\cite{CP74}) which we refer to as the Parry measure.
It is defined by
\begin{equation}\label{eq:SFT-measure}
 \mu (w_1, \ldots, w_n) = u_{i(w_1)} v_{t(w_n)} \lambda^{-(n-1)}.
\end{equation}

We remark that if $G$ is an irreducible graph, then the entropy of the associated 
 shift of finite type is $h(X^{G}) = \log( \lambda_{G})$.

\subsection{Linearly recurrent subshifts}
\label{ssec:LR}

Linear recurrence (or linear repetitivity) is a very strong regularity condition for a subshift.
The name refers to estimates on the size of the biggest gap between successive repetitions of patterns of a given size.
Equivalently, it refers to the biggest return times in small neighbourhoods of a given size in the dynamical 
system $(X,\mathbb Z)$.

\begin{defn}
If $w$ is a finite admissible word of a subshift $X$, a \emph{right return word} to $w$ of $X$ is any word $r$ such that: --~$rw$ is admissible; --~$w$ is a prefix of $rw$; --~there are exactly two occurrences (possibly overlapping) of $w$ in $rw$.
\end{defn}

\begin{defn}
A minimal subshift $X$ is \emph{linearly recurrent} if (or LR) there exists a constant $K$ such that for all $n \in \N$, any word $w \in X_{n}$ and any right return word $r$ to $w$, we have
\[
 \abs{r} \leq K \abs{w}.
\]
\end{defn}
It is a bound on the return time in the cylinders $U(w)$ for $w \in X_n$: for any $x \in U(w)$, there is $0 < k \leq K \abs{w}$ such that $\sigma^k(x) \in U(w)$.

Here are a few important consequences of linear recurrence.
\begin{prop}[Durand \cite{Dur00, Dur03}]
Let $X$ be a linearly recurrent subshift. Then $(X,\sigma)$ is uniquely ergodic.
\end{prop}

\begin{prop}[Durand \cite{Dur00, Dur03}]\label{prop:LR-measure-clopen}
Let $X$ be a linearly recurrent subshift with constant $K$, and $\mu$ be its unique invariant measure. Then for all $n \in \N$ and all $w \in X_n$,
\begin{equation}\label{LR-frequencies}
 \frac{1}{K} \leq n \mu(w) \leq K.
\end{equation}
\end{prop}
This presentation is actually somewhat backwards. In~\cite{Dur00}, Durand proves that equation~\eqref{LR-frequencies} holds for any invariant measure, and uses it in conjunction with a result of Boshernitzan~\cite{Bos92} to prove unique ergodicity.

\begin{prop}[Durand \cite{Dur00, Dur03}]
Let $(X,\sigma)$ be a linearly recurrent subshift with constant $K$. Then
\begin{itemize}
 \item Any right return word $r$ to $w \in X_n$ has length at least $\abs{r} > \abs{w}/K$;
 \item For any fixed $w \in X_n$, there are at most $K(K+1)^2$ return words to~$w$.
\end{itemize}
\end{prop}

Also proved in~\cite{Dur00}, is the fact that any linearly recurrent subshift has sub-linear complexity in the sense that $\# X_n \leq C n$ for some constant $C$.

Proposition~\ref{prop:LR-measure-clopen} has an immediate consequence.
\begin{lemma}\label{lem:R-LR}
 Let $X$ be a linearly recurrent subshift. Then there exists a constant $C > 1$ such that for all $w \in X_n^+$,
 \[
   C^{-1} \leq R(w) \leq C.
 \]
\end{lemma}

Note that the class of linearly recurrent subshifts contains the class of primitive substitution subshifts.
Many well-known examples, such as the Fibonacci or the Thue--Morse subshifts are given by primitive substitutions, hence are linearly recurrent, and are therefore covered by our results (see~\cite{Fog02}, for example for a survey on substitutions).

\subsection{Sturmian systems}
\label{ssec:sturmian}

There are many different equivalent ways to define a Sturmian sequence (see~\cite{Arn02}, for example).
We give here the definition as the coding of an irrational rotation of angle $\theta$ on a circle.
A crucial ingredient to identify properties satisfied by a Sturmian sequence is the continued fraction expansion of the number $\theta$. We remind the reader of a few of its properties.
The results on continuous fraction, which are stated here without proof, can be found in the classic book by Khinchin\footnote{Also transcribed ``Khinchine'' in the French-language publications.}~\cite{Khi97}.
To fix notations, $\lfloor \theta \rfloor$ (integer part of $\theta$) refers to the unique integer satisfying $\theta-1 < \lfloor \theta \rfloor \leq \theta$, and $\{\theta\} := \theta - \lfloor \theta \rfloor \in [0,1)$ refers to the fractional part of $\theta$.

 Let $\theta$ be a real number, which we assume irrational. Let $\theta_0 = \theta$ and $\theta_1 = \{\theta\}$, so that $\theta = a_0 + \theta_1$, with $a_0 \in \Z$ and $\{\theta \} \in [0,1)$.
 Then there are a unique $a_1 \in \N$ and $0 < \theta_2 < 1$ such that
 \[
  \theta = a_0 + \frac{1}{a_1+\theta_2},
 \]
 namely $a_1 = \lfloor 1/\theta_1 \rfloor$ and $\theta_2 = \{1/\theta_1\}$.
 
 By iteration, there is a sequence $a_2, a_3, \ldots$ of positive integers such that for all $n$,
 \[
  \theta = a_0 + \cfrac{1}{a_1 + 
                  \cfrac{1}{\ddots
                   \cfrac{1}{a_n+\theta_{n+1}}}}.
 \]
 The number $\theta$ is entirely determined by its sequence of \emph{partial quotients} $(a_n)_{n \geq 0}$, and we write $\theta = [a_0; a_1, a_2, \ldots]$.

 The finite continued fraction (truncating after the term $a_n$) 
 is called the $n$-th convergent and is denoted $[a_0; a_1, \ldots, a_n]$. It is a rational number, which has a unique representation as an irreducible fraction $p_n/q_n$.
 
 It is a standard result that $p_n$ and $q_n$ satisfy the recurrence relations
 \[
  \left\{ \begin{aligned}
   p_{n} & = a_n p_{n-1} + p_{n-2} \\
   q_{n} & = a_n q_{n-1} + q_{n-2}
  \end{aligned}\right.
  \quad
  \text{with the convention }
  \left\{ \begin{aligned}
     p_{-2} = 0, & \ \ p_{-1} = 1 \\
     q_{-2} = 1, & \ \ q_{-1} = 0 \, .
    \end{aligned}\right. 
 \]

\begin{lemma}\label{lem:cont-frac} 
 Let $0 < \theta < 1$ be irrational, so that $\theta = [0;a_1, a_2, \ldots]$. Let $(\theta_n)_{n\geq 0}$ be defined as above.
 We define $\lambda_n (\theta)$ (or just $\lambda_n$) by
 \(
  \lambda_n = \prod_{i=1}^n \theta_i
 \),
 for $n \geq 1$ (with the convention $\lambda_0 = 1$).
 Let $p_n/q_n$ be the $n$-th convergent.
 Then the following holds.
 \begin{enumerate}
  \item For all $n$,
   \(
    \theta = ({\theta_n^{-1} p_{n-1} + p_{n-2}})\big/({\theta_n^{-1} q_{n-1} + q_{n-2}}).
   \)
  \item For all $n$, $q_n p_{n-1} - p_n q_{n-1} = (-1)^n$.
  \item For all $n$, $\theta_n^{-1} q_{n-1} + q_{n-2} = \lambda_n^{-1}$.
  \item For all $n$, $\lambda_{n+1} =(-1)^n (\theta q_n - p_n)$.
 \end{enumerate}
\end{lemma}

\begin{proof}
 The first two points are classic (see for example Theorem~2 and the proof of Theorem~7 in~\cite{Khi97}). 
 The third point is proved by induction. For $n=1$, this is the equality $\theta_1^{-1} = \lambda_1^{-1}$.
 Assume it holds for fixed $n$.
Then
 \[ \begin{split}
  \theta_{n+1}^{-1} q_{n} + q_{n-1} & = \theta_{n+1}^{-1}\big( (a_n+ \theta_{n+1}) q_{n-1} + q_{n-2} \big) \\
     & = \theta_{n+1}^{-1}\big( \theta_n^{-1} q_{n-1} + q_{n-2} \big) \\
     & = \theta_{n+1}^{-1} \lambda_n^{-1} = \lambda_{n+1}^{-1}.
 \end{split}
 \]
 For the fourth point, we start apply points 1, 2 and 3 in succession:
 \[
  \begin{split}
   (-1)^n(\theta q_n - p_n) & = (-1)^n \frac{q_n(\theta_{n+1}^{-1} p_n + p_{n-1}) - p_n(\theta_{n+1}^{-1} q_n + q_{n-1})}{\theta_{n+1}^{-1} q_n + q_{n-1}} \\
    & = (-1)^n \frac{q_n p_{n-1} - p_n q_{n-1}}{\theta_{n+1}^{-1} q_n + q_{n-1}}
     = (-1)^{2n} \lambda_{n+1} = \lambda_{n+1}.
  \end{split}
 \]
\end{proof}

\begin{defn}
 Sturmian sequences of irrational parameter $0 < \theta < 1$ are defined as the sequences $w \in \{0,1\}^{\Z}$ such that
 \[
  w_n = \begin{cases}
   0 & \text{if } \{n \theta - x\} \in I, \\
   1 & \text{otherwise},
  \end{cases}
 \]
 where $x$ is a real number, $\{y\} = y- \lfloor y \rfloor$ is the fractional part of the real number $y$, and $I$ is either $[0,1-\theta)$ or $(0,1-\theta]$

 Let $X^\theta$ be the set of all Sturmian sequences of parameter $\theta$.
\end{defn}

It is well known that $X^\theta$ is a closed shift-invariant subspace of the full shift.
By irrationality of $\theta$, it is aperiodic and minimal.

\begin{prop}
 If $X^\theta$ is a Sturmian subshift with irrational $\theta$, then $\# X^\theta_n = n+1$. It implies that $\# (X_n^\theta)^+ = 1$ for all~$n$.
\end{prop}

It is well known that the subshifts $X^\theta$ are uniquely ergodic, and the frequency of the letter $1$ is equal to the length of the interval $I$, \emph{i.e.}\ $\theta$ (the frequency of the letter $1$ is defined as the measure of the cylinder set associated with the word $1 \in X^\theta_1$ for the unique invariant measure).

The measure of the cylinder sets for a Sturmian subshift is very well controlled in terms of the partial quotients.
This fact can be seen as a consequence of the three-distances theorem.
We use the three distance theorem the way it is presented in~\cite{AB98} (noting that our $\lambda_{n+1}$ is their $\eta_n$, by point~4 of Lemma~\ref{lem:cont-frac}). The relationship between this theorem and the frequencies of Sturmian sequences is established in~\cite{Ber96}.
\begin{thm}[\cite{AB98,Ber96}]\label{thm:freq-sturm}
 Let $X^\theta$ be a Sturmian subshift of parameter $0 < \theta < 1$ irrational.
 If $kq_n + q_{n-1} \leq m < (k+1)q_n +q_{n-1}$ for some $n \geq 1$ and $0 < k \leq a_{n+1}$, then the frequencies $\mu(w)$ ($w \in X^\theta_m$) are in the set:
 \[
  \bigg\{  \big(\lambda_n - k \lambda_{n+1} \big),\  \lambda_{n+1} ,\  \big( \lambda_n - (k-1) \lambda_{n+1} \big) \bigg\}.
 \]
 If $m= (k+1) q_n + q_{n-1} -1 $, the third frequency does  not occur.
\end{thm}

Remark in particular that if $m= q_n + q_{n-1} -1 = (a_n+1) q_{n-1} + q_{n-2} -1$, then the $\mu(w)$ is either $\lambda_{n}$ or $\lambda_{n-1} - a_n \lambda_n = \lambda_{n+1}$ for $w \in X_m^\theta$.

\section{Spectral triples}
\label{sec:ST}

From this point on, we assume that a subshift $(X,\sigma)$ is given. In time, we will assume that it is a shift of finite type, a linearly recurrent subshift, or a Sturmian subshift.

We consider the crossed product $C^{*}$-algebra $C(X) \times \Z$, which is generated by $f \in C(X)$ and a unitary $u$
satisfying $u f = f \circ \sigma^{-1} u$, for all $f$ in $C(X)$. This is represented on $L^{2}(X, \mu)$ by 
\begin{align*}
 (f \xi) (x) & = f(x) \xi(x), &
 (u \xi) (x) & = \xi(\sigma^{-1}(x)),
\end{align*}
for $\xi$ in $L^{2}(X, \mu)$, $f$ in $C(X)$ and $x$ in $X$.

Let $C_{0} = \C 1 $ be the constant functions and, 
recalling that $\xi(w)$ denotes the characteristic function 
of the set $U(w)$, $w \in X_{n}, n \geq 1$, let
 $C_{n} = \Span \{ \xi(w) \mid w \in X_{n} \}$, for
$ n \geq 1$. Notice that $C_{0} \subset C_{1} \subset C_{2} \cdots$, each $C_{n}$ is a finite-dimensional
subspace of $L^{2}(X, \mu)$ and a finite-dimensional $*$-subalgebra of $C(X)$. Moreover, the union of the 
$C_{n}$ is dense in both, with the appropriate norms. For convenience, let $C_{-1} = \{ 0 \}$.
Let $P_{n}$ denote the orthogonal projection of $L^{2}(X, \mu)$  on $C_{n}$, for each non-negative $n$.
We let $C_{\infty}(X) = \cup_{n = 1}^{\infty} C_{n}$. We will denote by 
$C_{\infty}(X) \times \Z$ the $*$-algebra of operators on $L^{2}(X, \mu)$ generated by $C_{\infty}(X)$ and $u$.

As we indicated above, our main interest is in the crossed product $C(X) \times \Z$. We have here a specific 
representation of this $C^{*}$-algebra, but it is a reasonable question to ask if it is faithful.
First, the representation of $C(X)$ on $L^{2}(X, \mu)$ is faithful since we assume our measure $\mu$ has support equal to $X$.
Secondly, we use the fact  that, if our system is topologically free, that is, for every $n$ in $\Z$, the set
$\{ x \in X \mid \sigma^{n}(x) = x \}$ has empty interior in $X$, then every non-trivial, closed, two-sided ideal
in $C(X) \times \Z$ has a non-trivial intersection with $C(X)$. (See \cite{EH67} for a precise statement.)
In our case, since the representation of $C(X)$ is faithful, so is that of $C(X) \times \Z$, provided our
system is topologically free. We make this assumption implicitly from now on and remark that it holds
in all of our examples.

We define an operator $D$ as follows. First, we choose a strictly increasing sequence of
positive real numbers, $\alpha_{n}, n \geq 0$, and define
the domain of our operator 
\[
 \mathcal{D} = \Bigl\{ \sum_{n \geq 0} \xi_{n} \mid \xi_{n} \in C_{n} \cap C_{n-1}^{\perp}, \ \sum_{n} (\alpha_{n} \Vert \xi_{n} \Vert)^{2} < \infty \Bigr\}.
\]
We define
\[
 D \xi = \alpha_{n} \xi, \ \xi \in C_{n} \cap C_{n-1}^{\perp}.
\]

\begin{lemma}
\label{lemma:Cn}
Let $n \geq 0$. For all $f$ in $C_{n}$, we have 
\begin{enumerate}
 \item $f C_{m} \subset C_{m}$, for all $m \geq n$,
 \item $f \mathcal{D} \subset \mathcal{D}$, 
\item $[D, f] | C_{n}^{\perp} = 0$.
\end{enumerate}
\end{lemma}

\begin{proof}
 For the first item, we know that $C_{n} \subset C_{m}$ whenever $m \geq n$ and that $C_{m}$ is a unital subalgebra so that
 $f C_{m} \in C_{n} C_{m} \subset C_{m} C_{m} = C_{m}$. For the second, suppose that 
 $\xi$ is in $C_{m} \cap C_{m-1}^{\perp}$, for some $m > n$.  $C_{m}$ and $C_{m-1}$
 are invariant under $C_{n}$. Hence so is $C_{m-1}^{\perp}$ and $C_{m} \cap C_{m-1}^{\perp}$.
 The second part follows from this.
 On the space $C_{m} \cap C_{m-1}^{\perp}$, $D$ is a scalar and hence commutes with $f$. 
 Taking direct sums over all $m > n$ yields the result. 
\end{proof}

\begin{thm}
\label{thm:spectral-triple}
Let $X$ be a subshift.
\begin{enumerate}
 \item  $(C_{\infty}(X) , L^{2}(X, \mu), D)$ is a spectral triple if and only if the sequence $\alpha_{n}, n \geq 0$ tends to 
 to infinity.
 \item $[D, u]$ extends to a bounded linear operator on $L^{2}(X, \mu)$ if  and only if $\alpha_{n} - \alpha_{n-1}, n \geq 1$ is bounded.
 \item  $(C_{\infty}(X) \times \Z , L^{2}(X, \mu), D)$ is a spectral triple if and only if the sequence $\alpha_{n}, n \geq 0$ tends to 
 to infinity and $\alpha_{n} - \alpha_{n-1}, n \geq 1$ is bounded.
\end{enumerate}
\end{thm}

We use here Connes' definition of a spectral triple (or $K$-cycle)~\cite[Chapter IV-2]{Con94} (see also~\cite[Def.~9.16]{G-BVF01}).
There is a variety of conditions that may be required on a spectral triple depending on the context.
For example, in the context of quantum spaces, Rieffel~\cite{Rie99,Rie04} requires that a certain metric induced by $D$
on the state space of the algebra coincides with the weak-* topology.
We make no such assumption at this stage, however we shall investigate this Connes metric in Section~\ref{sec:connes}.

We remark also that  $D$ is defined 
to be a positive operator. It follows that the associated
Fredholm module has a trivial class in $K$-homology and carries no index data.

\begin{proof}
 We see from the third part of Lemma~\ref{lemma:Cn}, that for any $f$ in $C_{\infty}(X)$,
 $f$ maps the domain of $D$ to itself and from the final part, $[D,f]$ obviously 
 extends to a bounded linear operator on $L^{2}(X, \mu)$. The final condition is that 
 $(1+D^{2})^{-1}$ is compact  is evidently equivalent to its eigenvalues
 converging to zero, and also to the sequence $\alpha_{n}, n \geq 0$ tending to 
 infinity.
 
 Fix a positive integer $n$ and let $\mathcal{P}_{n}$ be the partition of $X$
 given by the cylinder sets:
 $\{ U(w) \mid w \in X_{n} \}$. It is can easily be seen  that the partition
 \[
 \sigma(\mathcal{P}_{n}) = \{ \sigma(U(w)) \mid w \in X_{n} \}
 \]
 is clearly finer than $\mathcal{P}_{n-2}$ while it is coarser than $\mathcal{P}_{n+2}$. 
 (It is probably easiest to check the cases of $n$ even and odd separately.)
 From this it follows that  $C_{n-2} \subset uC_{n} \subset C_{n+2}$. Hence, we also have
$C_{n-2}^{\perp} \supset uC_{n}^{\perp} \supset C_{n+2}^{\perp}$ as well.  
Let $\xi$ be any vector in $C_{n} \cap C_{n-1}^{\perp}$. We know that 
$u \xi$ is in $C_{n+2} \cap C_{n-3}^{\perp}$ and on this space, which is invariant under $D$, 
we have $\alpha_{n-2} \leq D \leq \alpha_{n+2}$.
Now, we compute
\[
 [D, u] \xi = D u \xi - u D \xi  = Du\xi - \alpha_{n} u \xi = (D - \alpha_{n}) u \xi,
\]
and the conclusion follows.
\end{proof}

In particular cases, we will consider the following spectral triple.

\begin{defn}\label{def:DX}
 Let $X$ be a subshift and let $\mu$ be a $\sigma$ invariant probability measure with support $X$. The spectral triple $(C_\infty(X), L^2(X,\mu), D_X)$ is defined as above, associated with the sequence $\alpha_n=n$.
\end{defn}

\begin{thm}
Let $X \subseteq \mathcal{A}^{\mathbb{Z}}$ and
 $X' \subseteq \mathcal{A}'^{\mathbb{Z}}$ be subshifts with 
invariant probability measures $\mu$ and $\mu'$, respectively.
If there exists a homeomorphism 
$h : X \rightarrow X'$ such that
$h \circ \sigma|_{X} = \sigma|_{X'} \circ h$ and 
$h^{*}(\mu') = \mu$, then $v \xi = \xi \circ h$ 
is a unitary operator from $L^{2}(X', \mu')$ to 
$L^{2}(X, \mu)$ such that 
\begin{enumerate}
\item $v C_{\infty}(X') v^{*} = C_{\infty}(X)$,
\item $ v u_{X'} v^{*} = u_{X}$,
\item $v$ maps the domain of 
$D_{X'}$ onto the domain of $D_{X}$ and $vD_{X'}v^{*} - D_{X}$ is bounded.
\end{enumerate}

\end{thm}

\begin{proof}
That $v$ is unitary and the first two parts of the conclusion
all follow trivially from the hypotheses. For the last 
part, we use the result of Curtis--Lyndon--Hedlund~\cite[Theorem 6.2.9]{LM95} that 
$h$ arises from a sliding block code. That is, there
exist  positive integers $M,N$ and function 
$h_{0}: X_{N+M+1} \rightarrow \mathcal{A}'$ such that 
\[
h(x)_{n} = h_{0}(x_{n-M} \ldots x_{n+N}),
\]
for all $x$ in $X$. We lose 
no generality in assuming that $M=N$.
It then easily follows that for any $w$ in 
$X_{n}, n > N$, $h(U(w))$ is contained in 
$U'(w')$, where $w'$ is in $X_{n-2N}$ and is given by 
$w'_{i} = h_{0}(w_{i-N} \ldots w_{i+N})$, for all 
appropriate $i$. It follows that 
\[
v C'_{n-2N} v^{*} \subseteq C_{n}.
\]
 A similar result
also holds for $h^{-1}$ and we may assume without loss
of generality, it uses the same $N$. This means that 
\[
v C'_{n} v^{*} \supseteq C_{n-2N},
\]
as well. From these two facts, it is easy to see that 
$\Vert vD_{X'}v^{*} - D_{X} \Vert \leq 2N$.
\end{proof}

\section{Summability}

The goal of this section is to study the summability of the operator $D_{X}$ as introduced in Definition~\ref{def:DX}, for a shift space $X$.
For $p \geq 1$, the
 spectral triple $D_X$ is said to be
\emph{$p$-summable} if $(1+D_X^2)^{-p/2}$ is trace class; and \emph{$\theta$-summable} if for all $t>0$, $e^{-tD_X^2}$ is trace class. 

Recall the definitions from Section~\ref{sec:subshifts}, in particular Equation~\eqref{eq:entropy}.
Also note that in our construction earlier, $\dim (C_{n}) = \# X_{n}$, for all $n \geq 1$, while
$\dim(C_{n} \cap C_{n-1}^{\perp}) = \# X_{n} - \# X_{n-1}$, for all $n \geq 0$.

\begin{thm}
 \label{thm:theta-summable}
 Let $X \subset \mathcal{A}^{\Z}$ be a subshift and $D_{X}$ be defined as in Definition~\ref{def:DX}. 
 \begin{enumerate}
  \item If $s > h(X)$, then $e^{-s D_{X} } $ is trace class. 
  \item If $s > 0$ is such that $e^{-s  D_{X} }$ is trace class, then $s \geq h(X)$.
 \end{enumerate}
\end{thm}

\begin{proof}
 For the first part, choose $s >  s' > h(X)$. There exists $n_{0} \geq 1$ such that 
$\frac{1}{n} \log(\# X_{n}) < s'$, for all $n \geq n_{0}$. Hence, we have 
$\# X_{n} \leq e^{ns'}$ and $e^{-sn} \# X_{n}   \leq e^{n(s'-s)}$, for $n \geq n_{0}$. Since $s' < s$ the series
$e^{n(s'-s)}$ is a summable geometric series. It follows that $e^{-sn} \# X_{n} $ is also summable.
Also, we have
\[
 \Tr(e^{-sD_{X}})   =   \sum_{n =0}^{\infty} e^{-sn} \dim(C_{n} \cap C_{n-1}^{\perp}) 
                  \leq \sum_{n =0}^{\infty} e^{-sn} \dim(C_{n}). 
\]
Since $\dim (C_n) = \# X_n$, this completes the proof of the first part.

Conversely, suppose that  $e^{-s  D_{X} } $ is trace class. This operator is positive and its trace is
\[
 \Tr(e^{-s  D_{X}} ) = \sum_{n=0}^{\infty} e^{-sn}( \dim(C_{n}) - \dim(C_{n-1})).
\]
An Abel transform (summation by part) on the partial sum gives:
\[
  \sum_{n=0}^{N} e^{-sn}( \dim(C_{n}) - \dim(C_{n-1})) 
           \geq   \sum_{n=0}^{N-1} (1 - e^{-s} ) e^{-sn} \dim(C_{n}).
\]
Since the left-hand side converges, the series 
$\sum e^{-sn} \dim(C_{n})$ is convergent and 
hence the terms go to zero. Thus, we have 
\[
  0  = \lim_{n \rightarrow \infty} e^{-sn} \dim(C_{n}) 
     = \lim_{n \rightarrow \infty} e^{-sn + \log(\dim(C_{n}))}
     = \lim_{n \rightarrow \infty} \bigl( e^{-s + \frac{1}{n} \log(\# X_{n})} \bigr)^{n}
\]
We conclude that for $n$ sufficiently large, $e^{-s + \frac{1}{n} \log(\# X_{n})}$ 
is less than $1$. On the other hand we know that 
$\frac{1}{n}\log(\# X_{n} )$ tends to $h(X)$. The conclusion follows.
\end{proof}

\begin{thm}
 \label{thm:finite-summable}
  Let $X \subset \mathcal{A}^{\Z}$ be a subshift and $D_{X}$ be defined as in Definition~\ref{def:DX}.
  \begin{enumerate}
   \item  If, for some constants $C, s_{0} \geq 1$, 
   $\# X_{n} \leq C n^{s_{0}}$, for all $n \geq 1$, then $D_{X}$ is
  $s$-summable for all $s > s_{0} $.
  \item   If $D_{X}$ is
  $s$-summable for some $s > 0$, then  there exists $C \geq 1$ such that 
  $\# X_{n} \leq C n^{s}$, for $n \geq 1$.
  \end{enumerate}
\end{thm}

\begin{proof}
 For the first part, let $s > s_0$ and compute
\begin{multline*}
  \Tr\bigl( (1+ D_{X}^{2})^{-s/2} \bigr) =  \lim_{N \rightarrow \infty} 
                   \sum_{n=0}^{N} (1 + n^{2})^{-s/2} \big( \dim(C_{n}) - \dim(C_{n-1}) \big)  \\
       =  \lim_{N \rightarrow \infty} \bigg[ \sum_{n=1}^{N} \Big( \big(1 + n^{2} \big)^{-s/2} - \big(1 + (n+1)^{2} \big)^{-s/2}  \Big) \dim(C_{n}) \\
           + \big(1 + N^{2}\big)^{-s/2}  \dim(C_{N}) \bigg] 
\end{multline*}
The second term $(1 + N^{2})^{-s/2}  \dim(C_{N})$ tends to zero
from the hypothesis. For the first term, we denote
 \[
   S_N := \sum_{n=1}^{N}  \bigl( (1 + n^{2})^{-s/2} - (1 + (n+1)^{2})^{-s/2}  \bigr) \dim(C_{n})
 \] 
for convenience.
Consider the function $f(x) = -(1 + x^{2})^{-s/2} $.
Its derivative  is $f'(x) =  s x(1 + x^{2})^{-s/2 - 1}$ which is positive and decreasing for $x \geq 1$.
By an application of the mean value theorem to this function on the interval $[n, n+1]$, we obtain
\[
 \big(1 + n^{2}\big)^{-s/2} - \big(1 + (n+1)^{2} \big)^{-s/2}  \leq s n ( 1 + n^{2})^{-s/2 -1} \leq sn^{-s-1}.
\]
We multiply by $\dim(C_{n})$ and sum. Since $s > s_0$, we obtain
\[
   S_{N}   \leq  s \sum_{n=1}^{N} n^{-s-1} \dim(C_{n})    
           \leq  s \sum_{n=1}^{N} C n^{s_0-s-1}
          <    \infty.
\]

 For the second statement, assume $D_X$ is $s$-summable for a fixed $s$, and
let $\beta_{n}, n \geq 0$ be any $\ell^{1}$ sequence.
We denote $r_n := \sum_{m \geq n} \abs{\beta_m}$.
Define a bounded operator $B$ on 
$L^{2}(X, \mu)$ by setting it equal to $r_n$ on the space 
$C_{n} \cap C_{n-1}^{\perp}$, for all $n \geq 0$. Notice that $B$ is positive and commutes with $D_{X}$.
The operator $(1+D_X^{2})^{-s/2}B$ is then a positive, trace class operator so, for every $N \geq 1$,
we have  (line 2 to line 3 is a summation by parts):
\begin{align*}
 \infty & >    \Tr \Big( (1+D_{X}^{2})^{-s/2}B \Big) \\
        & \geq    \sum_{n=0}^{N} (1 + n^{2})^{-s/2}  
                  r_n \big( \dim(C_{n}) - \dim(C_{n-1}) \big)  \\
        & \geq \sum_{n=0}^{N-1} \dim(C_{n}) 
                  \bigl[ r_n (1 + n^{2})^{-s/2} - r_{n+1}  (1 + (n+1)^{2})^{-s/2} \bigr] \\
        & \geq \sum_{n=0}^{N-1} \dim(C_{n}) 
                  \bigr[ (r_n -r_{n+1}) (1 + n^{2})^{-s/2}   \bigl] 
         = \sum_{n=0}^{N-1} \dim(C_{n}) 
                 \bigr[ \vert \beta_{n} \vert (1 + n^{2})^{-s/2} \bigl].
\end{align*}
We have shown that for any $(\beta_{n})_{n \geq 0} \in \ell^1$, the sequence
$(\beta_{n} \dim(C_{n}) (1 + n^{2})^{-s/2})_{n \geq 0}$ is also in $\ell^{1}$. It follows that
$(\dim(C_{n}) (1 + n^{2})^{-s/2})_{n \geq 0}$ is bounded.
The conclusion follows after noting that $(1 + n^{2})^{s/2} \leq C n^{s}$, for some constant $C$ and all $n \geq 1$.

\end{proof}

\section{The Connes metric}
\label{sec:connes}

In this section, we develop some tools to analyze the Connes metric and apply them to our specific cases of interest.
These results are very technical so it may be of some value to indicate, at least vaguely, what is going on.
It seems that a number of factors influence the behaviour of the Connes metric for a subshift $X$.
The two most important ones are:
\begin{enumerate}
 \item For a given $x$ in $X$, the sequence $\{ n \in \N \mid \pi_{n}(x) \in X_{n}^{+} \}$. 
 That is, in looking at the sequence of words $\pi_{n}(x)$, how often do we have strict 
 containment $U(\pi_{n}(x)) \supsetneq U(\pi_{n+1}(x))$? Generally speaking, large gaps 
 between such values of $n$ will tend to help keep the Connes metric finite. For example, in 
 linearly recurrent (in particular substitutive) subshifts, this sequence tends to grow exponentially (see
 the discussion after Lemma~\ref{lem:LR-bound-Rw}).
 \item For the various $w$, the ratio ${\mu(\pi(w))}/{\mu(w)}$. These ratios remaining bounded
 seem to help keep the Connes metric finite.
\end{enumerate}

Our main interest is in the operator $D_{X}$ whose eigenvalues are the natural numbers. However, we will treat
the more general case of $D$ for most of our estimates, partly so that we can see the exact
r\^{o}le of the eigenvalues $\alpha_{n}$. We will assume that this sequence is increasing and unbounded.

\subsection{Technical preliminaries}
\label{ssec:connes-technical}

For $w$ in $X^{+}_{n}$, we introduce a space of functions 
\[
 F_{w} = \Span \big\{ \xi(w') \mid \pi(w') = w \big\} \cap C_{n}^{\perp}.
\]
In fact, we could also make the same definition for any $w$ in $X_{n}$, 
but $F_{w}$ is non-zero if and only if $w$ is in $X_{n}^{+}$.
This is a finite-dimensional space and it will be convenient to have an estimate
comparing the supremum norm with the $L^{2}$-norm for its elements.

\begin{lemma}
\label{lemma:L2-Linfinity}
 For $w$ in $X^{+}_{n}$ and $f$ in $F_{w}$, we have 
 \[
  \Vert f \Vert_{2} \leq \mu(w)^{\frac{1}{2}}  \Vert f \Vert_{\infty}  \leq R(w)^{\frac{1}{2}}  \Vert f \Vert_{2}.
 \]
\end{lemma}

\begin{proof}
The first inequality is trivial. For the second, since $f$ is in $F_{w}$, $f= \sum_{\pi(w')=w} a_{w'} \xi(w')$
and so $ \Vert f \Vert_{\infty} = | a_{w'} |$, for some $w'$ with $\pi(w') = w$. Then we have 
\[
 \Vert f \Vert_{2}^{2} = \int_{U(w)} f^{2} d\mu \geq |a_{w'}|^{2} \mu(w') =  \mu(w') \Vert f \Vert_{\infty}^{2}
 \geq \mu(w) R(w)^{-1} \Vert f \Vert_{\infty}^{2}.
\]
Multiplying by $R(w)$ and taking square roots yields the result.

\end{proof}

We are next going to introduce a special collection of functions. At the simplest level they provide 
spanning sets for our subspaces $F_{w}$. More importantly and subtly, their interactions with the operator $D$
are quite simple. 

Suppose that $\pi(w)$ is in $X^{+}_{n}$, let 
\[
 \eta(w) = \frac{\mu(\pi(w))}{ \mu(w)}  \xi(w) - \xi(\pi(w) ).
\]
Notice that we could make the same definition for any $w$, but if $\pi(w)$ is not in 
$X^{+}_{n}$, for some $n$, then $\eta(w)=0$. Also, if $w$ is in $X^{+}_{n}$, the
vectors  $\{ \eta(w') \mid \pi(w')=w \}$ span $F_{w}$, but are linearly dependent since their
sum is zero. 
Also observe that for any $w$ in $X_{n}$, the sequence of vectors $\eta(\pi_{m+1}(w))$, for those $m < n$ with
$\pi_{m}(w)$ in $X_{m}^{+}$ are an orthogonal set.

We summarize some properties of  $\eta(w)$.

\begin{lemma}
\label{lemma:eta}
 Suppose that $\pi(w)$  is  in $X^{+}_{n}$.
\begin{enumerate}
 \item $\eta(w)$ is in $C_{n+1}$.
\item $\eta(w)$ is in $C_{n}^{\perp}$.
\item $\eta(w)^{2} = \left( \frac{\mu(\pi(w)) }{ \mu(w)} - 2 \right)\eta(w) 
 +  \left( \frac{\mu(\pi(w)) }{ \mu(w)} - 1 \right) \xi(\pi(w)) $.
\item $\Vert \eta(w) \Vert_{2}^{2} = \mu(\pi(w)) \left( \frac{\mu(\pi(w)) }{ \mu(w)} - 1 \right)$.
\item If we consider $\eta(w)$ in $C(X)$ and $\eta(w) + \C$ in the quotient  space $C(X)/ \C$, then
$\Vert \eta(w) + \C \Vert_{\infty} =   \frac{\mu(\pi(w)) }{2 \mu(w)}$.
\end{enumerate}

\end{lemma}

\begin{proof}
 The first part is clear. For the second, $\eta(w)$ is clearly supported on $\xi(\pi(w))$ so it 
 is perpendicular to $\xi(w')$, for all $w' \neq \pi(w)$ in $X_{n}$. In addition, we compute 
 \[
  \langle \eta(w), \xi(\pi(w)) \rangle =  \frac{\mu(\pi(w)) }{ \mu(w)}  \mu(w) - \mu(\pi(w)) = 0.
 \]
For the third part, use the fact that $\xi(w)$ and $\xi(\pi(w))$ are idempotents and their 
product is the former, so that
\[\begin{split}
  \eta(w)^{2} & =  \left( \frac{\mu(\pi(w))^{2}}{\mu(w)^{2}} - 2 \frac{\mu(\pi(w))}{\mu(w)} \right) \xi(w) + \xi(\pi(w)) \\
              & =  \left( \frac{\mu(\pi(w)) }{ \mu(w)} - 2 \right)  \frac{\mu(\pi(w))}{\mu(w)}  \xi(w) + \xi(\pi(w)) \\
              & =  \left( \frac{\mu(\pi(w)) }{ \mu(w)} - 2 \right) (\eta(w) + \xi(\pi(w))) + \xi(\pi(w)) \\
              & =  \left( \frac{\mu(\pi(w)) }{ \mu(w)} - 2 \right) \eta(w)  +  \left( \frac{\mu(\pi(w)) }{ \mu(w)} - 1 \right) \xi(\pi(w)).
\end{split}\]

For the fourth part, we know already that $ \frac{\mu(\pi(w)) }{ \mu(w)} \xi(w) = \eta(w) + \xi(\pi(w))$. 
Moreover, the two terms on the right are orthogonal, so we have 
\[
 \left( \frac{\mu(\pi(w)) }{ \mu(w)}\right)^{2} \mu(w) = \Vert \eta(w) \Vert^{2}_{2} + \mu(\pi(w)),
\]
and the conclusion follows.

For the last part, the function $\eta(w)$ takes on three values: $0$, 
$a = \frac{\mu(\pi(w)) }{ \mu(w)} - 1  > 0$ (on $U(w)$) and 
 $b = -1 < 0$.
 For any such function, the minimum norm on $\eta(w) + \C$ is obtained 
 at $\eta(w) - \frac{a+b}{2}$ and equals $a - \frac{a+b}{2} = \frac{a -b}{2}$. The rest is a 
 simple computation.
\end{proof}

We next want to summarize some of the properties of the functions $\eta(w)$ as operators. In 
particular, their commutators with $D$ are of a simple form.

\begin{lemma}
\label{lemma:zeta}
 Suppose that $w$ is in $X_{n+1}$ with $\pi(w) $ in $X^{+}_{n}$. Considering $\eta(w)$ as an element of $C(X)$, we have 
the following. 
\begin{enumerate}
 \item $\eta(w) C_{n} \subset \C \eta(w) \subset C_{n+1} \cap C_{n}^{\perp}$.
\item $[\eta(w), D] = \eta(w) \otimes \zeta(w)^{*} - \zeta(w) \otimes \eta(w)^{*}$,
where \newline $\zeta(w) = \mu(\pi(w))^{-1}(D - \alpha_{n+1})\xi(\pi(w))$.
\item $\Vert [ \eta(w), D ] \Vert = \Vert \eta(w) \Vert_{2} \Vert \zeta(w) \Vert_{2}$.
\item $\zeta(w)$ is in $C_{n}$ and is non-zero.
\item $\Vert [ \eta(w), D ] \Vert \leq  \left( \frac{\mu(\pi(w)) }{ \mu(w)} \right)^{1/2} \alpha_{n+1}  $.
\end{enumerate}

\end{lemma}

\begin{proof}
 For the first point, the support of $\eta(w)$ is contained in $U(\pi(w))$.
 So for any $w'$ in $X_{n}$,  $\eta(w) \xi(w') = 0$ unless $w' = \pi(w)$,
 in which case  $\eta(w) \xi(w') =  \eta(w) \in C_{n+1} \cap C_n^\perp$.

For the second point, we know that $D C_{n} \subset C_{n}$ and $D(C_{n+1} \cap C_{n}^{\perp})$ is a subspace of $(C_{n+1} \cap C_{n}^{\perp})$. 
Combining this with the first part above, we see that 
$D \eta(w)$ and $\eta(w)D$ both map $C_{n}$ into $(C_{n+1} \cap C_{n}^{\perp})$.
We also know that $\restr{[\eta(w), D]}{C_{n+1}^{\perp}} = 0$. In other words, $[\eta(w),D](I-P_{n+1})=0$.
From above, we have $P_{n}[\eta(w), D]P_{n} = 0$. In addition, we compute
\begin{align*}
( & P_{n+1} - P_{n}) [\eta(w), D] ( P_{n+1} - P_{n}) \\
  &       \qquad =  ( P_{n+1} - P_{n}) (\eta(w) D - D \eta(w)) ( P_{n+1} - P_{n}) \\
  &       \qquad = ( P_{n+1} - P_{n}) \eta(w) D  ( P_{n+1} - P_{n}) -  ( P_{n+1} - P_{n}) D \eta(w)  ( P_{n+1} - P_{n}) \\
  &       \qquad =  ( P_{n+1} - P_{n}) \eta(w) \alpha_{n+1} ( P_{n+1} - P_{n}) \\*
  &   \qquad \quad -  ( P_{n+1} - P_{n}) \alpha_{n+1} \eta(w)  ( P_{n+1} - P_{n}) =0 
\end{align*}
From all this---writing $I=P_n + (P_{n+1}-P_n) + (I-P_{n+1})$---, we have 
\[ 
[\eta(w), D] = (P_{n+1}-P_{n}) [\eta(w), D] P_{n} + P_{n} [\eta(w), D] (P_{n+1}-P_{n}) 
\]
Since $[\eta(w), D]$ is skew adjoint, the second term is the opposite of the first. From part 1, the 
first is a rank one operator of the form $\eta(w) \otimes \zeta(w)^{*}$, for some 
$\zeta(w)$ in $C_{n}$.

To find $\zeta(w)$, the simplest thing is to apply $[\eta(w), D]$ to $\eta(w)$ so that
\begin{align*}
 &\Vert \eta(w)\Vert_{2}^{2} \zeta(w)
          =  - [\eta(w), D](\eta(w)) \\
 & \qquad =  - \eta(w) D(\eta(w)) - D(\eta(w)^{2}) \\
 & \qquad =  - \alpha_{n+1} \eta(w)^{2} - D(\eta(w)^{2}) \\
 & \qquad =  - (\alpha_{n+1} - D) (\eta(w)^{2}) \\
 & \qquad =  - (\alpha_{n+1} - D)  
    \left[  \left( \frac{\mu(\pi(w)) }{ \mu(w)} - 2 \right)\eta(w) 
 +  \left( \frac{\mu(\pi(w)) }{ \mu(w)} - 1 \right) \xi(\pi(w))  \right] \\
 & \qquad =    0 + \left( \frac{\mu(\pi(w)) }{ \mu(w)} - 1 \right)  (D - \alpha_{n+1}) \xi(\pi(w)). 
\end{align*}
Now using the fact that $\Vert \eta(w)\Vert_{2}^{2} = \left( \frac{\mu(\pi(w)) }{ \mu(w)} - 1 \right)  \mu(\pi(w))$, we have
\[\begin{split}
 \zeta(w)   & =  \Vert \eta(w)\Vert_2^{-2} \left( \frac{\mu(\pi(w))}{\mu(w)} - 1 \right)  (D - \alpha_{n+1}) \xi(\pi(w)) \\
            & =  \mu(\pi(w))^{-1} (D - \alpha_{n+1}) \xi(\pi(w)).
\end{split}\]

For the third part, having
expressed  $[\eta(w), D] = \eta(w) \otimes \zeta(w)^{*} - \zeta(w) \otimes \eta(w)^{*}$,
we note that the first term maps $C_{n}$ to $C_{n}^{\perp}$ and the second maps
$C_{n}^{\perp}$ to $C_{n}$. Hence the norm is equal to the supremum of the norms of
the two terms which is simply $\Vert \eta(w) \Vert_{2} \Vert \zeta(w) \Vert_{2}$.

The first statement of the fourth part follows from the fact that $\xi(\pi(w))$ is in $C_{n}$, which is invariant under $D$,
and the formula for $\zeta(w)$. The second part follows because $D- \alpha_{n+1}$ is invertible  on $C_{n}$ as
$\alpha_{n+1} > \alpha_{1}, \alpha_{2}, \ldots, \alpha_{n}$.

For the last part, we have
\begin{align*}
 & \Vert [\eta(w), D] \Vert = \Vert \eta(w) \Vert_2 \Vert \zeta(w) \Vert_2  \\
    & \quad = \left( \frac{\mu(\pi(w))}{\mu(w)} - 1 \right)^{1/2} \mu(\pi(w))^{1/2}
                                 \mu(\pi(w))^{-1} \Vert (D- \alpha_{n+1}) \xi(\pi(w)) \Vert_{2}  \\
    &  \quad \leq \left( \frac{\mu(\pi(w))}{\mu(w)}\right)^{1/2}
                                  \mu(\pi(w))^{-1/2} \Vert (D - \alpha_{n+1}) \xi(\pi(w)) \Vert_{2}.
\end{align*}
It remains only to note that $\mu(\pi(w))^{1/2} = \Vert \xi(\pi(w)) \Vert$ and that $\xi(\pi(w))$ is a vector in
$C_{n}$, with  $ 0 \leq D|_{C_{n}} \leq \alpha_{n+1}$ so that $\Vert D - \alpha_{n+1} \Vert \leq \alpha_{n+1}$.

\end{proof}

A nice simple consequence of this result is that the only functions in $C_{\infty}$ commuting with $D$ are the scalars.
We introduce now the notation $Q_{w}$ for the orthogonal projection of $L^{2}(X, \mu)$ onto $F_{w}$.
It will be convenient to break a function $f \in L^2(X, \mu)$ into its components $Q_w(f)$.
Since the commutator of $D$ with the functions $\eta(w')$ has a fairly simple form and $Q_w$ is the projection onto the span of the
of $\eta(w') $ for $\pi(w') = w$, we can hope that the quantities $[D,Q_w(f)]$ are accessible to computations.
Note that  $Q_{w}(f) = \xi(w)(P_{n+1}(f) - P_{n}(f))$, for any $f$ in 
$L^{2}(X, \mu)$.

\begin{thm}
 \label{thm:commutator-of-D}
 Let $X$ be a subshift and assume  the sequence $\alpha_{n}$ is strictly increasing. 
 If $f$ is in $C_{\infty}(X)$ and $[f, D] = 0$, then $f$ is in $C_{0} = \C 1$.
\end{thm}

\begin{proof}
 Begin by adding a scalar to $f$ so that $f $ is in $C_{0}^{\perp}$ which clearly alters neither hypothesis nor
 conclusion. Suppose that $f$ is in $C_{n}$. We will prove by induction on $n$ that $f=0$. The case
 $n=0$ is straightforward, since $f$ lies in both $C_{0}$ and $C_{0}^{\perp}$.
 Now assume $n \geq 1$.
 We may write
 \[
  f = \sum_{0 \leq m < n} \sum_{ \pi(w) \in X_{m}^{+}} \beta(w) \eta(w),
 \]
 where the $\beta(w)$ are scalars. 
 
 Fix $w_{0}$ in $X_{n-1}^{+}$ and let  $w_{1} $ be in $X_{n}$ with $\pi(w_{1}) = w_{0}$. 
 We will compute $[f, D] \eta(w_{1})$, 
 which is zero by hypothesis. This will be done term by term in the sum above for  $f$, using
 part~2 of Lemma~\ref{lemma:zeta}.
 First, suppose that $w$ is in $X_{m}$ with $m < n$. In this case, both $\eta(w)$ and $\zeta(w)$ lie in $C_{n-1}$
 and hence their inner products with $\eta(w_{1})$ are zero so this term contributes nothing.
 
 Secondly, if $w$ is in $X_{n}$ with $\pi(w) \neq w_{0}$, then $\eta(w)$ is 
 supported on $U(\pi(w))$ while $\eta(w_{0})$ is supported on $U(w_{0})$. As these sets are 
 disjoint, the vectors are orthogonal. Since $\zeta(w)$ is in $C_{n-1}$, it is also orthogonal
 to $\eta(w_{0})$. 
 
 This leaves us to consider $w$ with $\pi(w) = w_{0}$. For each such $w$, $\zeta(w)$ is again in
 $C_{n-1}$ and is orthogonal to $\eta(w_{1})$. We conclude that 
 \[
  [f, D] \eta(w_{1}) = \sum_{\pi(w) = w_{0}} \beta(w) \langle \eta(w_{1}), \eta(w) \rangle \zeta(w).
 \]
Observe from the formula in part 2 that $\zeta(w)$ depends only on $\pi(w)=w_{0}$ and is non-zero. We conclude that
\[
 \sum_{\pi(w) = w_{0}} \beta(w) \langle \eta(w_{1}), \eta(w) \rangle = 0.
\]
Since $\eta(w)$ and $\eta(w_{1})$ are real-valued, we also have 
\begin{align*}
 0  & = \sum_{\pi(w) = w_{0}} \beta(w) \langle \eta(w), \eta(w_{1}) \rangle
      = \Big\langle  \sum_{\pi(w) = w_{0}} \beta(w) \eta(w), \eta(w_{1}) \Big\rangle \\
    & = \langle Q_{w_{0}}(f) , \eta(w_{1}) \rangle.
\end{align*}
If we now vary $w_{1}$ over $\pi^{-1}\{ w_{0}  \}$, this forms a spanning set for
$F_{w_{0}}$. We conclude that $Q_{w_{0}}(f) = 0$. As this holds for every $w_{0}$ in
$X_{n}^{+}$, we conclude that $f$ actually lies in $C_{n-1}$. By induction hypothesis
$f=0$.
\end{proof}

It is a good time to observe the following fairly simple consequence of these estimates. We will use this
later to give examples of Sturmian systems where the Connes metric is infinite.

\begin{thm}
 \label{thm:sufficient-for-unbounded}
 Let $X$ be a subshift and let $\mu$ be a 
$\sigma$-invariant probability measure with support $X$. If the set
 \(
  \left\{ (n+1)^{-2} {\mu(\pi(w))}/{\mu(w)} \mid n \geq 1, \pi(w) \in X_{n}^{+} \right\}
 \)
is unbounded, the Connes metric associated with $(C_{\infty}(X), L^{2}(X, \mu), D_{X})$ is infinite.
\end{thm}

\begin{proof}
 It follows from the last part of Lemma~\ref{lemma:zeta} that the functions 
 \[
 f_{w} = (n+1)^{-1} \left( \frac{\mu(\pi(w))}{\mu(w)} \right)^{-1/2} \eta(w)
 \]
 all satisfy  $\Vert [f_{w}, D] \Vert \leq 1$. On the other hand, from 
 part 5 of Lemma \ref{lemma:eta}, we have 
 \[
  \Vert f_{w} + \C \Vert_{\infty} = (2(n+1))^{-1} \left( \frac{\mu(\pi(w))}{\mu(w)} \right)^{1/2},
 \]
 which, by hypothesis, is unbounded.
\end{proof}

The next results will be used as tools in showing the the Connes metric is finite 
and induces the weak-$*$ topology for some subshifts. The point is really to get estimates 
on how a function may move vectors in 
$C_{n} \cap C_{n-1}^{\perp}$ to $C_{m} \cap C_{m-1}^{\perp}$, for values $n > m$.
The second is the important estimate; we include the first because it seems
convenient to divide it into two steps.

\begin{lemma}
\label{lem:control-Qw}
  Let $f$ be in $\cup_{n} C_{n}$. Let $w$ be in $X^{+}_{n}$ and $m<n$ with $\pi_{m}(w)$ in $X^{+}_{m}$. 
We have 
\[
 \big\langle [D, f] \eta\big(\pi_{m+1}(w)\big), Q_{w}(f) \big\rangle  = (\alpha_{n+1} - \alpha_{m+1}) \Vert Q_{w}(f) \Vert _{2}^{2} 
   \left( \frac{\mu(\pi_{m}(w))}{\mu(\pi_{m+1}(w)} -1 \right).
\]
\end{lemma}

\begin{proof} 
For convenience let $w' = \pi_{m+1}(w)$.
Observe that the function $Q_{w}(f)$ is in $C_{n+1}\cap C_{n}^{\perp}$ while $\eta(w')$ is in $C_{m+1}\cap C_{m}^{\perp}$ .
Therefore, we have
\begin{align*} 
\langle [D, f] \eta(w'), Q_{w}(f) \rangle
        & =  \langle Df \eta(w'),  Q_{w}(f) \rangle - \langle fD \eta(w'), Q_{w}(f) \rangle   \\
        & =  \langle f \eta(w'), D Q_{w}(f) \rangle - \langle \alpha_{m+1}f\eta(w'), Q_{w}(f) \rangle  \\ 
        & =  \langle f \eta(w'), \alpha_{n+1} Q_{w}(f) \rangle - \langle \alpha_{m+1}f\eta(w'), Q_{w}(f) \rangle  \\
        & =  (\alpha_{n+1} - \alpha_{m+1}) \langle  f \eta(w'), Q_{w}(f) \rangle.
\end{align*}
Now we observe that $Q_{w}(f)$ is supported on $U(w) \subset U(w')$ (since $w' = \pi_{m+1}(w)$) 
and $\eta(w')$ is identically $ \left( \frac{\mu(\pi(w'))}{\mu(w')} -1 \right) $ on the latter. Therefore, we have 
\begin{align*}
 \langle [D, f] \eta(w'), Q_{w}(f)) \rangle
       & = (\alpha_{n+1} - \alpha_{m+1}) \langle f \eta(w'), Q_{w}(f) \rangle   \\
       & = (\alpha_{n+1} - \alpha_{m+1}) \left( \frac{\mu(\pi(w'))}{\mu(w')} -1 \right) \langle f, Q_{w}(f) \rangle  \\
       & = (\alpha_{n+1} - \alpha_{m+1}) \left( \frac{\mu(\pi(w'))}{\mu(w')} -1 \right) \langle Q_{w}(f), Q_{w}(f) \rangle.
\end{align*}

\end{proof}

\begin{prop}
\label{prop:control-Qw}
 Let $f$ be in $\cup_{n} C_{n}$.  Let $w$ be in $X^{+}_{n}$ and $m<n$ with $\pi_{m}(w)$ in $X^{+}_{m}$.
 If $\Vert [D, f] \Vert \leq 1$, then 
 \[
  \Vert Q_{w}(f) \Vert_{\infty}  \leq    R(w)^{1/2} R(\pi_{m}(w) )^{1/2} 
  (\alpha_{n+1} - \alpha_{m+1})^{-1}  \mu(w)^{-1/2} \mu(\pi_{m+1}(w))^{1/2} . 
 \]
\end{prop}

\begin{proof}
Let $w' = \pi_{m+1}(w)$.
From the hypothesis and the last Lemma, we have 
 \begin{align*}
 \Vert  \eta(w') \Vert_{2} \Vert Q_{w}(f) \Vert_{2}
              & \geq  \abs{ \big\langle [D, f] \eta(w'), Q_{w}(f) \big\rangle } \\
              &  =   (\alpha_{n+1} - \alpha_{m+1}) \Vert Q_{w}(f) \Vert_{2}^{2}  \left( \frac{\mu(\pi(w'))}{\mu(w')} -1 \right),
 \end{align*}
and hence 
\(
\Vert  \eta(w') \Vert_{2}  \geq (\alpha_{n+1} - \alpha_{m+1}) \Vert Q_{w}(f) \Vert_{2} 
 \left( \frac{\mu(\pi(w'))}{\mu(w')} -1 \right)
\).
Now we invoke our estimate for $\Vert  \eta(w') \Vert_{2} $  from  Lemma~\ref{lemma:eta} and 
recall from Lemma~\ref{lemma:Rw} that
\(
 \left( \frac{\mu(\pi(w'))}{\mu(w')} -1 \right)^{-1} \leq R(\pi(w')) = R(\pi_{m}(w))
\),
so we have
\begin{align*}
\Vert Q_{w}(f) \Vert_{\infty}
       &  \leq \mu(w)^{-1/2} R(w)^{1/2} \Vert Q_{w}(f) \Vert_{2} \\
       &  \leq \mu(w)^{-1/2} R(w)^{1/2}  \Vert  \eta(w') \Vert_{2}
          (\alpha_{n+1} - \alpha_{m+1})^{-1} \left( \frac{\mu(\pi(w'))}{\mu(w')} -1 \right)^{-1} \\
       &  =    \mu(w)^{-1/2} R(w)^{1/2}  \mu(w')^{1/2} \left( \frac{\mu(\pi(w')}{\mu(w')} -1 \right)^{1/2} \\*
   & \qquad \qquad  \qquad \qquad      (\alpha_{n+1} - \alpha_{m+1})^{-1} \left( \frac{\mu(\pi(w'))}{\mu(w')} -1 \right)^{-1} \\
       & \leq  \mu(w)^{-1/2} R(w)^{1/2}  \mu(w')^{1/2}
            (\alpha_{n+1} - \alpha_{m+1})^{-1} \left( \frac{\mu(\pi(w'))}{\mu(w')} -1 \right)^{-1/2} \\
       &  =    \mu(w)^{-1/2} R(w)^{1/2}  \mu(\pi_{m+1}(w))^{1/2} 
             (\alpha_{n+1} - \alpha_{m+1})^{-1}  R(\pi_{m}(w))^{-1/2}.
\end{align*}
The conclusion follows.
\end{proof}

\subsection{Shifts of finite type}
\label{ssec:connes-SFT}

\begin{thm}
 \label{thm:connes-SFT}
 Let $G$ be an aperiodic, irreducible graph, $X^{G}$ be its associated shift of finite type and $\mu$ be the Parry measure. 
 The Connes metric from $(C_{\infty}(X^{G}), L^{2}(X^{G}, \mu), D_{X^{G}})$ is infinite.
\end{thm}

The proof will involve some technical lemmas. Let us first lay down some basic tools.
Recall that we write our vertex set $G^{0} = \{ 1, 2, \ldots, I \}$. Also recall that $u_{G}, v_{G}$ 
are the left and right 
Perron eigenvectors for the incidence matrix $A$ with Perron eigenvalue $\lambda_{G}$,
appropriately normalised.
For convenience, we shorten these to $u, v$ and $\lambda$ respectively.
The invariant measure is defined by Equation~\eqref{eq:SFT-measure}.

 Since $G$ is aperiodic there exists a vertex $j$  with 
 $\# t^{-1}\{ j \} > 1$. 
 Let us assume that $j=1$. It follows then
 that for any word $w= e_{1}e_{2} \cdots e_{n}$ of even length with $i(e_{1}) = 1$,  $w$ is in 
 $X^{+}_{n}$.
 
 From the fact that $G$ is irreducible, we may find a word $w= e_{1}e_{2} \cdots e_{p}$ with 
 $i(e_{1}) = 1 = t(e_{p})$. 
 For each $k \geq 1$, we define $w_{k} = e_{p}w^{2k}$. Note that $w_{k}$ is in 
 $X_{2kp + 1}$ and $\pi(w_{k}) = w^{2k}$, which is in  $X^{+}_{2kp}$. 
 
 We note that $\mu(w_{k}) = \Vert \xi(w_{k}) \Vert_{2} ^{2} = v_{i(e_{p})}u_{1} \lambda^{-2kp}$ and 
 that $\mu(\pi(w_{k})) = \Vert \xi(\pi(w_{k})) \Vert_{2} ^{2} = v_{1}u_{1} \lambda^{-2kp+1}$.
 It follows that 
 \begin{equation}\label{eq:a}
 1 <  \frac{\mu(\pi(w_{k}))}{\mu(w_{k}) } = v_{1}v_{i(e_{p})}^{-1} \lambda.
 \end{equation}
 This is constant in $k$ and we denote its value by $a$ for convenience.

 For $K \geq 1$, consider the function $f_{K} = \sum_{k=1}^{K} k^{-1} \eta(w_{k})$. 
 Notice this function
 is in $C_{0}^{\perp}$, so it takes on both positive and 
 negative values. Secondly, from the definition, we see that
 \[
  f_{K} | U(w_{k}) = \sum_{k=1}^{K} k^{-1} \left( \frac{\mu(\pi(w_{k}))}{ \mu(w_{k})} - 1 \right)
     = ( a - 1 ) \sum_{k=1}^{K} k^{-1}.
 \]
Since $f_{K}$ takes on negative values, we conclude that in the quotient $C(X)/\C$, we have 
$\Vert f_{K} + \C \Vert_{\infty} \geq \frac{a-1}{2} \sum_{k=1}^{K} k^{-1}$. The obvious
conclusion being that these norms go to infinity as $K$ does. We will now show that 
$\Vert [ f_{K}, D] \Vert$ is bounded and the conclusion follows from \cite[Theorem 2.1]{Rie04}.

It follows from part~3 of Lemma~\ref{lemma:zeta} that
 \[ \begin{split}
 [f_{K}, D] & = \sum_{k=1}^{K} k^{-1} \big( \eta(w_{k}) \otimes \zeta(w_{k})^{*} - \zeta(w_{k}) \otimes  \eta(w_{k}) ^{*} \big) \\
            & = \left( \sum_{k=1}^{K} k^{-1}  \zeta(w_{k}) \otimes \eta(w_{k})^{*} \right)^{*} - \sum_{k=1}^{K} k^{-1}  \zeta(w_{k}) \otimes \eta(w_{k})^{*}.
 \end{split} \]
  Let $T_{K} = \sum_{k=1}^{K} k^{-1}  \zeta(w_{k}) \otimes \eta(w_{k})^{*}$. It suffices
  to prove that $\Vert T_{K} \Vert$ is bounded, independent of $K$.
  
  First notice that $T_{K}$ is zero on the orthogonal complement of the
  span of the $\eta(w_{k}), 1 \leq k \leq K$. Secondly, this set
  is orthonormal.
  Therefore, in order to control $\nr{T_K}$, it suffices to evaluate the norm of products $\nr{T_K \eta}$ for vectors $\eta$ of the form
  \[
   \eta = \sum_{k=1}^{K} \beta_{k} \mu(\pi(w_{k}))^{-1/2} (a-1)^{-1/2} \eta(w_{k}).
  \]
It follows from Lemma~\ref{lemma:eta} and~Equation~\eqref{eq:a} that $\Vert \eta \Vert_{2}^{2} = \sum_{k=1}^{K} \vert \beta_{k} \vert^{2}$.
 
 We also have 
\begin{align*}
  T_{K} \eta & = \left[ \sum_{k=1}^{K} k^{-1}  \zeta(w_{k}) \otimes \eta(w_{k})^{*} \right] 
                     \sum_{k'=1}^{K} \beta_{k'} \mu(\pi(w_{k'}))^{-1/2} (a-1)^{-1/2} \eta(w_{k'}) \\
             & = \sum_{k=1}^{K} \sum_{k'=1}^{K} k^{-1} \beta_{k'} \mu(\pi(w_{k'}))^{-1/2} (a-1)^{-1/2}  \langle \eta(w_{k'}) , \eta(w_{k}) \rangle \zeta(w_{k}) \\
             & = \sum_{k=1}^{K} k^{-1}  \beta_{k} \mu(\pi(w_{k}))^{-1/2} (a-1)^{-1/2} \mu(\pi(w_{k}))  (a-1)\zeta(w_{k}) \\
             & = \sum_{k=1}^{K} k^{-1} \beta_{k} \mu(\pi(w_{k}))^{1/2} (a-1)^{1/2} \zeta(w_{k}).
\end{align*}

\begin{lemma}\label{lemma:technical-SFT}
 For all $k \geq 1$, we have
 \begin{enumerate}
  \item $P_{2p(k-1) } \xi(\pi(w_{k})) = \lambda^{-2p} \xi(\pi(w_{k-1}))$,
  \item $\xi(\pi(w_{k})) - \lambda^{-2p} \xi(\pi(w_{k-1}))$ is in $C_{2pk} \cap C_{2p(k-1)}^{\perp}$,
  \item $\Vert \xi(\pi(w_{k})) - \lambda^{-2p} \xi(\pi(w_{k-1})) \Vert^{2}_{2} = u_{1} v_{1} (\lambda - \lambda^{1-2p}) \lambda^{-2pk}$.
  \item
   \(\displaystyle
    \xi(\pi(w_{k})) = \sum_{j=1}^{k} \lambda^{-2p(k-j)} \left[ \xi(\pi(w_{j})) -  \lambda^{-2p} \xi(\pi(w_{j-1})) \right] + \lambda^{-2pk}
   \).
 \end{enumerate}
\end{lemma}

 \begin{proof}
  The vectors $\xi(w), w \in X_{2p(k-1)}$ form an orthonormal set in $C_{2p(k-1)}$.
  The only one of these with a non-zero inner product with $\xi(\pi(w_{k}))$ is 
  $\xi(\pi(w_{k-1}))$. Hence we have 
  \begin{align*}
     P_{2p(k-1)} \xi(\pi(w_{k})) & = 
           \frac{ \langle \xi(\pi(w_{k})), \pi(\xi(w_{k-1})) \rangle}{ \langle \xi(\pi(w_{k-1})), \pi(\xi(w_{k-1})) \rangle} \xi(\pi(w_{k-1})) \\
       & = \frac{\mu(\xi(\pi(w_{k})))}{\mu( \xi(\pi(w_{k-1})))} \xi(\pi(w_{k-1})) 
        = \lambda^{-2p} \xi(\pi(w_{k-1})).
  \end{align*}

  The second part follows immediately from the first. The third follows from the orthogonality of $\lambda^{-2p} \xi(\pi(w_{k-1}))$
  and $\xi(\pi(w_{k})) - \lambda^{-2p} \xi(\pi(w_{k-1}))$, so we have 
  \begin{align*}
   \Vert \xi(\pi(w_{k})) - \lambda^{-2p} \xi(\pi(w_{k-1})) \Vert^{2}_{2} 
     & =  \Vert \xi(\pi(w_{k})) \Vert^{2}_{2} - \Vert  \lambda^{-2p} \xi(\pi(w_{k-1})) \Vert^{2}_{2} \\
     & =  u_{1}v_{1}\lambda^{-2kp+1} - \lambda^{-4p}u_{1}v_{1}\lambda^{-2(k-1)p+1}  \\
     & =  u_{1}v_{1} (\lambda  - \lambda^{1-2p}) \lambda^{-2kp}.
   \end{align*}

   The last also follows from the first and an easy induction argument, which we omit.
  \end{proof}

In the following computation, we denote
$\xi \big( \pi(w_{j}) ) -  \lambda^{-2p} \xi(\pi(w_{j-1}) \big)$ by
$\xi_{j}$, for $ 1 \leq j \leq K$.
We compute
\begin{align*}
 T_{K} \eta
    & = \sum_{k=1}^{K} k^{-1} \beta_{k} \mu(\pi(w_{k}))^{1/2} (a-1)^{1/2} \zeta(w_{k}) \\
    & = \sum_{k=1}^{K} k^{-1} \beta_{k} \mu(\pi(w_{k}))^{-1/2} (a-1)^{1/2} (2pk+1 - D) \xi(\pi(w_{k})) \\
    & = \sum_{k=1}^{K} k^{-1} \beta_{k} (u_{1}v_{1} \lambda^{-2pk+1})^{-1/2} (a-1)^{1/2} (2pk+1 - D) \xi(\pi(w_{k})) \\
    & = \left( \frac{a-1}{u_{1}v_{1} \lambda} \right)^{1/2}  \sum_{k=1}^{K} k^{-1} \beta_{k}\lambda^{pk} (2pk+1 - D) \\*
        & \qquad \left[ \sum_{j=1}^{k} \lambda^{-2p(k-j)} ( \xi(\pi(w_{j})) -  \lambda^{-2p} \xi(\pi(w_{j-1})))  + \lambda^{-2pk} \right] \\
    & = \left( \frac{a-1}{u_{1}v_{1} \lambda} \right)^{1/2} \sum_{k=1}^{K} \sum_{j=1}^{k}  k^{-1} \beta_{k}  \lambda^{-pk}\lambda^{2pj} (2pk+1 - D) \xi_{j} \\*
    &  \qquad  + \left( \frac{a-1}{u_{1}v_{1} \lambda} \right)^{1/2}\sum_{k=1}^{K} k^{-1} \beta_{k} \lambda^{pk}(2pk+1 - D)\lambda^{-2pk}.  \\
\end{align*}
Let us begin by considering the second term on the right which is simply a constant function considered
as a vector in $L^{2}(X, \mu)$.  Ignoring the constants in front of the sum, the norm of this vector is bounded by
\begin{align*}
     \sum_{k=1}^{K} \vert  \beta_{k} \vert    k^{-1} (2pk+1 )\lambda^{-pk}    
                      &  \leq  \sum_{k=1}^{K} \vert  \beta_{k} \vert   (2p+1) \lambda^{-pk}  \\  
                      &  \leq (2p+1) \left( \sum_{k=1}^{K} \vert  \beta_{k} \vert^{2}  \right)^{1/2}  \left( \sum_{k=1}^{K} \lambda^{-2pk}  \right)^{1/2} \\
                      &  \leq  (2p+1)  \left( \frac{\lambda^{-2p}}{1 - \lambda^{-2p}} \right)^{1/2} \Vert \eta \Vert_{2}.
\end{align*}

Now we turn to the first term.
For any given $j$, the vector $\xi_{j}$ is in $C_{2pj} \cap C_{2p(j-1)}^{\perp}$
and on this space $1 + 2p(j-1) \leq D \leq 2pj$.
For values of $k \geq j$, we have $0 \leq 2pk +1 -D \leq 2p(k-j+1)$
and, in particular, $\Vert  2pk +1 -D  \Vert \leq 2p(k-j+1)$.
Fixing $k$ for the moment, the norm of the vector $\sum_{j=1}^{k} \lambda^{-pk}\lambda^{2pj} (2pk+1 - D) \xi_{j}$ is bounded by:
\begin{multline*} 
  \sum_{j=1}^k \lambda^{2pj} \lambda^{-pk}(k-j+1)\bigl( u_1 v_1 (\lambda - \lambda^{1-2p}) \lambda^{-2pk} \bigr)^{1/2} \\
   \leq C \sum_{j=1}^k \lambda^{-2p(k-j)} (k-j+1) 
   = C \sum_{j=1}^k j \lambda^{-2p(j-1)}
\end{multline*}
where $C$ gathers the constants and the equality is a change of variable.
Now, by deriving the power series $\sum x^n$ and letting $x=\lambda^{-2p}$ we get
\[
 \sum_{j=1}^k j\lambda^{-2p(j-1)} \leq \frac{1}{(1-\lambda^{-2p})^2}.
\]
If we now sum over $k$ and apply the Cauchy--Schwarz inequality, we get
\begin{align*}
 \biggl\Vert \sum_{k=1}^{K} \sum_{j=1}^{k}  k^{-1} \beta_{k}  \lambda^{-pk}\lambda^{2pj} (2pk+1 - D) \xi_{j} \biggr\Vert
   & \leq C' \sum_{k=1}^K \abs{\beta_{k}} k^{-1} \\
   & \leq C' \nr{\beta}_2 \Bigl({\sum_{k \geq 1} k^{-2}}\Bigr)^{1/2} \leq C'' \nr{\beta}_2.
\end{align*}

Putting all of this together we have shown that $\Vert T_{K} \eta \Vert_{2} $ is bounded above by a constant, independent of $K$, times  $\Vert  \eta \Vert_{2} $.
This complete the proof of Theorem~\ref{thm:connes-SFT}.

\subsection{Linearly recurrent subshifts}

\begin{thm}
 \label{thm:LR}
Let $X$ be a linearly recurrent subshift and 
let  $\mu$ be  its unique invariant measure.
Then the  spectral triple
\[
(C_{\infty}(X), L^{2}(X, \mu), D_{X})
\]
has finite Connes distance and the topology induced on the state space of $C(X)$ is the weak-$*$ topology.
\end{thm}

The proof will occupy the remainder of the subsection, and use Proposition~\ref{prop:control-Qw} in a crucial way.

The general idea of the proof goes along these lines:
starting from a function $f \in C_\infty (X)$ with $\nr{ [f,D]} \leq 1$,
Proposition~\ref{prop:control-Qw} allows to control $\nr{Q_w(f)}_\infty$ where $w \in X_n^+$, in terms of~$n$.
Since any such function can be written up to a constant term as
$f=\sum_{w \in X^+} Q_w(f)$, this allows to control the norm of~$f$ in the algebra $C(X) / \mathbb C$.
Eventually, Rieffel's criterion~\cite{Rie99,Rie04} is applied.

First, this lemma is an immediate consequence of Proposition~\ref{prop:LR-measure-clopen}.

\begin{lemma}\label{lem:LR-bound-Rw}
 Let $X$ be a linearly recurrent subshift.
 There are $C_1, C_2 > 0$ such that for all $n$ and all $w \in X_n$, $C_1 \leq R(w) \leq C_2$.
\end{lemma}

Let us consider now $w \in X_n^+$. Any such $w$ is the prefix of a word of the form $rw$ where $r$ is a return word to~$w$.
Therefore, all sufficiently long right-extensions of $w$ start by a word of the form $rw$, with $r$ some return-word to be picked in a finite set. Besides, the length of $rw$ is always at least $(1+1/K) n$.
Therefore, $w$ has at most $K(K+1)^2$ right-extensions of length $(1+1/K)n$.
As a conclusion, there are at most $K(K+1)^2$ indices $n < k \leq (1+1/K)n$ for which the right-extension of length $k$ of $w$ (say $w'$) is a special word.

A similar result holds for left-extensions.
The conclusion is then that for any $w \in X_n$, the number of indices $n < k \leq n(1+1/K)$ for which there is a $w' \in \pi_n^{-1}(\{w\}) \cap X_k^+$ is bounded above, and this bound is uniform in $n$ (it only depends on $K$).

Consider now a sequence $f_i \in C_\infty$, such that $\nr{[D,f_i]} \leq 1$ for all $i$.
The goal is to extract a converging subsequence from the image of the sequence $(f_i)_i$ in the quotient algebra $C(X)/\mathbb C$.
If we assume that all the $f_i$ belong to $C_0^\perp$ (which is possible up to adding a suitable scalar multiple of the identity), it is then sufficient to prove that the sequence $(f_i)_i$ has a convergent subsequence in $C(X)$.

Since $f_i \in C_0^\perp$, we can write
\[
 f_i = \sum_{w \in X^+} Q_w (f_i).
\]
For a fixed $w$ in $X_{n}^{+}$, $ n \geq 1$, we may apply Proposition~\ref{prop:control-Qw} with $m=0$
to conclude that
\begin{equation}\label{eq:Qw-LR}
 \begin{split}
  \nr{Q_{w}(f_{i})}_{\infty} & \leq  R(w)^{1/2} R(\varepsilon)^{1/2} (n+1)^{-1} \mu(w)^{-1/2} \\
                             & \leq C_0 (n+1)^{-1} \mu(w)^{-1/2},
 \end{split}
\end{equation}
where $C_0$ is a constant which does not depend on $w$.
In particular, this is bounded independent of $i$. Since the vector space $F_{w}$
is finite-dimensional, any bounded sequence must have a convergent subsequence.

Enumerate $X^{+} = \{ w_{1}, w_{2}, w_{3}, \ldots \}$.  
Let $(f^{1}_{n})_{n \geq 1}$ be a subsequence of $(f_{n})_{n \geq 1}$ such that $Q_{w_{1}}(f^{1}_{n})$ converges.
Inductively define a sequence $(f^{k}_{n})_{n \geq 1}$ as a subsequence
of $(f^{k-1}_{n})_{n \geq 1}$ such that $(Q_{w_{k}}(f^{k}_{n}))_n$ converges.

We now prove that the sequence $(f^n_n)_{n \geq 1}$ converges in $C(X)$.
It is enough to show that it is Cauchy.
The goal is to show that $\nr{f_n^n(x)-f_m^m(x)}_\infty$ goes to zero uniformly in $x$ when $n,m$ tend to infinity.

To prepare this, consider a function $f$ and evaluate
\(
 \sum_{w \in X^+} Q_w(f) (x)
\)
for $x \in X$. For a given $x$, infinitely many terms of the sum are zero.
It is possible to rewrite it: if $f \in C_0^\perp$ and $\nr{[D,f]} \leq 1$,
\[
 \abs{f(x)}  = \Bigl\vert \sum_{n \in \mathbb N} \mathbf{1}_{ \{\pi_n(x) \in X_n^+\}} (n) Q_{\pi_n(x)}(f)(x) \Bigr\vert.
\]
In particular,
\begin{equation}\label{conv-fx}
 \Bigl\vert f(x) - \sum_{n=1}^N \sum_{w \in X_n^+} Q_w(f) (x) \Bigr\vert \leq \sum_{n > N} \mathbf{1}_{\{\pi_n(x) \in X_n^+\}} (n) \nr{Q_{\pi_n(x)}(f)}_\infty.
\end{equation} 
By the argument above, the cardinality of the set
\[
 \bigl\{ n \in [(1+1/K)^k, (1+1/K)^{k+1})  \mid  \pi_n(x) \in X_n^+ \bigr\}
\]
is bounded above by $K(K+1)^2$, independently of $k$ and $x$. Call $I_k$ the interval $[(1+1/K)^k, (1+1/K)^{k+1})$. One has
\begin{multline}\label{conv-fx-bis}
 \Bigl\vert f(x) - \sum_{n=1}^N \sum_{w \in X_n^+} Q_w(f) (x) \Bigr\vert 
   \leq \sum_{k \geq k_0} K(K+1)^2 \max \big\{ \nr{Q_{\pi_n(x)}(f)}_\infty \ : \\  \ n \in I_k \text{ such that } \pi_n(x) \in X_n^+ \big\}.
\end{multline}

Now, it turns out that if $n \in I_k$, then $\mu(w)^{-1/2} \leq K^{1/2} (1+1/K)^{(k+1)/2}$, and $1/n \leq (1+1/K)^{-k}$, so using the bound of Equation~\eqref{eq:Qw-LR} with $w = \pi_n(x)$ together with Lemma~\ref{lem:R-LR}, we get
\[
 \nr{Q_{\pi_n(x)}(f)}_\infty \leq C \alpha^k, \quad \text{for } \alpha = \Big( 1+\frac{1}{K}\Big)^{-1/2} < 1.
\]
It is crucial to notice that both constants $C$ and $\alpha$ are completely independent of $f$ and $x$ (as long as $\nr{[D,f]} \leq 1$).
In conclusion, the convergence in~\eqref{conv-fx-bis} is geometric, and
\(
 \sum_{w \in X^+} Q_w(f)
\)
converges uniformly to $f$ on $X$. In particular, its norm converges.

Back to the sequence $(f^n_n)_{n \geq 1}$. Let $\eps > 0$. First, choose $M$ such that $\nr{\sum_{n \geq M}\sum_{w \in X_n^+} Q_w(f)}_\infty$ is smaller than $\eps/4$ (independent of $f$).
Next, number the elements of $\bigcup_{i < M} X_i^+ := \{ w_1, \ldots, w_L \}$. For each $1 \leq l \leq L$, pick a number $I_l$ such that for all $n,m \geq I_l$,
\[
 \nr{Q_{w_l}(f_n^l) - Q_{w_l} (f_m^l)}_\infty \leq \frac{\eps}{2L}.
\]
Pick $I = \max \{I_l \mid 1 \leq l \leq L\}$.
Then, for all $n,m \geq I$, we can compute
\begin{align*}
  & \bigl\Vert{f_n^n - f_m^m}\bigr\Vert_\infty = \Bigl\Vert{ \sum_{w \in X^+} Q_{w}(f_n^n - f_m^m) }\Bigr\Vert_\infty \\
  &   \quad  \leq \Bigl\Vert{ \sum_{l=1}^L Q_{w_l} (f_n^n - f_m^m) }\Bigr\Vert_\infty 
           + \Bigl\Vert{\sum_{\substack{n \geq M \\ w \in X_n^+}} Q_{w}(f_n^n)}\Bigr\Vert_\infty 
           + \Bigl\Vert{\sum_{\substack{n \geq M \\ w \in X_n^+}} Q_{w}(f_m^m)}\Bigr\Vert_\infty \\
  &   \quad  \leq \sum_{l=1}^L \Bigl\Vert{ Q_{w_l} (f_n^n) - Q_{w_l} (f_m^m) }\Bigr\Vert_\infty
                        + \Bigl\Vert{\sum_{\substack{n \geq M \\ w \in X_n^+}} Q_{w}(f_n^n)}\Bigr\Vert_\infty
                        + \Bigl\Vert{\sum_{\substack{n \geq M \\ w \in X_n^+}} Q_{w}(f_m^m)}\Bigr\Vert_\infty \\
  &   \quad  \leq L \frac{\eps}{2L} + \frac{\eps}{4} + \frac{\eps}{4} = \eps.
\end{align*}
Therefore, the sequence is Cauchy, so it converges in $C(X)$.

\subsection{Sturmian systems}

In this section, it is understood that the subshift in question is $X=X^\theta$.
The following result is a consequence of Theorem~\ref{thm:freq-sturm}.

\begin{lemma}\label{lem:sturm}
 Let $X^\theta$ be a Sturmian subshift with $0 < \theta < 1$ irrational.
 Let $\theta=[0;a_1, a_2, \ldots]$ be the continued fraction expansion of $\theta$.
 Then:
 \begin{enumerate}
  \item Assume $q_n + q_{n-1} - 1 \leq m < q_{n+1} + q_n - 1$. Then for all $w \in X_m^+$, $R(w) \leq a_{n+1}+1$.
  \item For any $n$, there exist $q_n + q_{n-1} \leq m < q_{n+1} + q_{n}$ and $w$ in $X_{m}$ such that ${\mu(\pi(w))}/{\mu(w)} \geq a_{n+1}$.
  \item Let $x \in X$. Then, for all $n$,
  \[
   \# \big\{ m \in \N \cap [q_n + q_{n-1} , q_{n+1} + q_n) \mid \pi_m (x) \in X_m^+ \big\} \leq  a_{n+1} + 2
  \]
  \item $q_n \geq 2^{(n-1)/2}$.
 \end{enumerate}
\end{lemma}

\begin{proof}
Point 1 is obtained by direct inspection of the formulas in Theorem~\ref{thm:freq-sturm}.
If $w \in X_m^+$ with $q_n + q_{n-1} \leq m < q_{n+1} + q_n - 1$, then the biggest quotient of the form $\mu (w) / \mu(w')$ for $\pi(w')=w$ is of the form
\[
 \frac{\lambda_n - k' \lambda_{n+1}}{\lambda_{n+1}} \leq \frac{\lambda_n}{\lambda_{n+1}} < a_{n+1}+1
\]
for some $0 \leq k' \leq a_{n+1}$.
If $m= q_{n} + q_{n-1} - 1 = (a_n + 1)q_{n-1}+q_{n-2}$, such a quotient is of the form
\(
 {\lambda_{n}}/{\lambda_{n+1}} \leq a_{n+1}+1
\).

 Similarly for point 2, observe that the frequencies associated with words of length $q_n+q_{n-1}-1$ are $\lambda_{n+1}$ and $\lambda_n$, while the frequencies associated with words of length $q_n + q_{n-1}$ are $\lambda_n$, $\lambda_{n+1}$ and $\lambda_n - \lambda_{n+1}$.
 Since $\# X_m^+ = 1$ for all $m$, let $w'$ be the unique word in $X^+_{q_n+q_{n-1}-1}$.
 By inspection, $\mu(w') = \lambda_n$ and it has two extensions of frequencies $\lambda_{n+1}$ and $\lambda_n - \lambda_{n+1}$. Let $w$ be its extension of frequency $\lambda_{n+1}$. Then $\mu(w')/ \mu(w) = \theta_{n+1}^{-1} \geq a_{n+1}$.

 Point 3 can also be deduced from Theorem~\ref{thm:freq-sturm}: $w \in X_m^+$ if and only if $\mu(w) \neq \mu(w')$ for $w' \in \pi^{-1}(\{w\})$.
 Now, $\mu(w)$ can only take $a_{n+1} + 2$ values if $w \in X_m$ and $q_n + q_{n-1} \leq m < q_{n+1} + q_n$.
 This proves the result.
 
 Point 4 is classic (see~\cite[Theorem~12]{Khi97}).
\end{proof}

Let us begin with the less positive result, simply because it is easier to prove.

\begin{thm}
 \label{thm:sturm-unbounded}
 There exists $0 < \theta <  1 $  irrational such that the  spectral triple 
$(C_{\infty}(X^{\theta}), L^{2}(X^{\theta}, \mu), D_{X^{\theta}})$
has infinite Connes metric. 
\end{thm}

\begin{proof}
Choose $a_{1}= 1$. Define $a_{n}$ inductively by
$a_{n+1} = n (q_{n} + q_{n-1})^{2}$. Letting $m$ and $w$ be as in part 2 of Lemma~\ref{lem:sturm},
we have 
\[
 (m+1)^{-2} \frac{\mu(\pi(w))}{\mu(w)}    \geq 
 (m+1)^{-2} a_{n+1} \geq (q_{n+1} + q_{n})^{-2} a_{n+1} = n.
\]
The conclusion follows from Theorem~\ref{thm:sufficient-for-unbounded}. 
\end{proof}

On the more encouraging side, we prove the following.

\begin{thm}
 \label{thm:sturm-bounded-ae}
 Let $0 < \theta <  1 $ be irrational with $\theta =  [0;a_{1}, a_{2}, \ldots ]$. If there exist
 constants $C \geq 1$, $s \geq 1$ such that $a_{j} \leq C j^{s}$, then 
 the  spectral triple 
$(C_{\infty}(X^{\theta}), L^{2}(X^{\theta}, \mu), D_{X^{\theta}})$
has finite Connes metric and the topology induced on the state space
of $C(X^{\theta})$ is the weak-$*$ topology.
\end{thm}

This result has an immediate consequence.

\begin{cor}
 For almost all $\theta \in (0,1)$ (for the Lebesgue measure), the spectral triple $(C_{\infty}(X^{\theta}), L^{2}(X^{\theta}, \mu), D_{X^{\theta}})$
 has finite Connes metric and the topology induced on the state space of $C(X^{\theta})$ is the weak-$*$ topology.
\end{cor}

\begin{proof}
 By a result of Khinchin~\cite[Theorem 30]{Khi97},
 if $\phi$ is a positive function defined on the natural numbers, then the estimate $a_n = O(\phi(n))$ holds for almost all $\theta \in (0,1)$ if and only if $\sum_{n \geq 1}(\phi(n))^{-1}$ converges.
 In this case, we consider the function $\phi(n) = n^s$ for $s > 1$.
 It follows that for almost all $\theta \in (0,1)$, the numbers $a_n$ satisfy $a_n \leq C n^s$.
 In particular, almost all $\theta \in (0,1)$ satisfy the hypotheses of Theorem~\ref{thm:sturm-bounded-ae}.
\end{proof}

 We now turn to the proof of Theorem~\ref{thm:sturm-bounded-ae}.
 It is actually quite similar to that of Theorem \ref{thm:LR} for linearly recurrent subshifts.

As for linearly recurrent subshift, if $f \in C_0^\perp$, we estimate
\begin{equation}\label{eq:unif-conv}
 \Bigl\vert{f(x) - \sum_{n=1}^N \sum_{w \in X_n^+} Q_w(f) (x)}\Bigr\vert \leq \sum_{n > N} \mathbf{1}_{\{\pi_n(x) \in X_n^+\}} (n) \nr{Q_{\pi_n(x)}(f)}_\infty.
\end{equation}
Now, we partition the positive integers into intervals $[q_k + q_{k-1} - 1, q_{k+1} + q_{k} - 1)$.
For $w \in X_n^+$ with $n$ in this interval, we can use Proposition~\ref{prop:control-Qw} together with points 1 and 2 of Lemma~\ref{lem:sturm} to get the following estimate.
This estimate holds for any $m$ in $[q_{k-2} + q_{k-3}-1, q_{k-1} + q_{k-2} - 1)$ such that $\pi_m(w) \in X_m^+$.
\[
 \begin{split}
 \nr{Q_w (f)}_\infty & \leq R(w)^{1/2} R(\pi_m(w))^{1/2} (n-m)^{-1} \mu(w)^{-1/2} \mu(\pi_{m+1}(w))^{1/2}   \\
                     & \leq (a_{k+1} a_k)^{1/2} (n-m)^{-1} \mu(w)^{-1/2} \mu(\pi_{m+1}(w))^{1/2}.
 \end{split}
\]
Now, given the intervals chosen,
\[
 (n-m)^{-1} \leq (q_k - q_{k-2})^{-1} = (a_k q_{k-1})^{-1} \leq 2^{-(k/2-1)} \leq \alpha^k
\]
for some $\alpha < 1$.
Furthermore, it results directly from Theorem~\ref{thm:freq-sturm} that $\mu(w)^{-1/2} \leq \lambda_{k+1}^{-1/2}$, and $\mu(\pi_{m+1}(w))^{1/2} \leq \lambda_{k-2}^{1/2}$.
Therefore, we have
\[
  \nr{Q_w (f)}_\infty  \leq (a_{k+1} a_k)^{1/2} \alpha^k (\theta_{k-1} \theta_k \theta_{k+1})^{-1/2}.
\]
It is straightforward from the definition that $\theta_k^{-1} \leq a_k + 1$. So using the assumption that $a_k = O(k^s)$, we can write
\[
 \nr{Q_w (f)}_\infty = O\big( k^{5s/2} \alpha^k \big).
\]
In other words, this is bounded above by a geometric series in $k$.
The rest is similar to the proof for linearly recurrent subshifts: using point 3 of Lemma~\ref{lem:sturm}, we prove that the expressions in Equation~\eqref{eq:unif-conv} are dominated by
\[
 \sum_{k \geq k_0} k^s (k^{5s/2} \alpha^k),
\]
for some $k_0 = k_0(N)$. It is essentially the remainder of a geometric series, which does not depend on $x$. Therefore the convergence of $\sum_w Q_w(f)$ to $f$ is uniform.
The rest of the proof is identical as before.

\bibliographystyle{abbrv}
\bibliography{biblio}

\end{document}